\documentclass[11pt,a4paper]{article}
\usepackage[latin1]{inputenc}
\usepackage{amsmath}
\usepackage{amsfonts}
\usepackage{amssymb}
\usepackage{graphicx}
\usepackage{geometry}
\usepackage{enumerate}
\usepackage{hyperref}
\usepackage{tikz-cd}
\hypersetup{
	colorlinks=true,        
	linkcolor=blue,         
	citecolor=blue,         
	urlcolor=blue           
}
\usepackage{dcolumn}
\usepackage{fancyvrb}
\usepackage{subcaption}
\usepackage{rotating}

\usepackage[ruled]{algorithm2e}

\usepackage{amsthm}
\usepackage{thmtools}

\newtheorem{theorem}{Theorem}[section]
\newtheorem{proposition}{Proposition}[section]
\newtheorem{corollary}{Corollary}[section]

\newtheorem{lemma}{Lemma}[section]
\newtheorem{problem}{Problem}[section]

\DeclareMathOperator{\Tr}{Tr}
\DeclareMathOperator{\Fix}{Fix}
\def\CsigmaMn{({\mathbb C}^{2^n\times 2^n})^{Q_M^n}}

\usepackage{todonotes}
\begin{document}
\thispagestyle{empty}
\title{A Douglas--Rachford construction of non-separable\\continuous
compactly supported \\multidimensional wavelets }
\author{David Franklin\footnotemark[1]
	   \and
	   Jeffrey A. Hogan\thanks{School of Mathematical and Physical Sciences,
	   	                   Mathematics Building, University of Newcastle,
	   	                   University Drive, Callaghan NSW 2308, \textsc{Australia}.	   	                   
	   	                  (\href{mailto:david.franklin@newcastle.edu.au }{david.franklin@newcastle.edu.au},
	   	                   \href{mailto:jeff.hogan@newcastle.edu.au}{jeff.hogan@newcastle.edu.au})}	                  
	   \and
	   Matthew K. Tam\thanks{School of Mathematics \& Statistics,
	                         The University of Melbourne,
	                         Parkville VIC 3010, \textsc{Australia}.	                    
	   	                    (\href{mailto:matthew.tam@unimelb.edu.au}{matthew.tam@unimelb.edu.au})}}

\maketitle
	\begin{abstract}
After re-casting the $n$-dimensional wavelet construction problem as a feasibility problem with constraints arising from the requirements of compact support, smoothness and orthogonality, the Douglas--Rachford algorithm is employed in the search for one- and two-dimensional wavelets. New one-dimensional wavelets are produced as well as genuinely non-separable two-dimensional wavelets in the case where the dilation on the plane is the standard $D_af(t)=a^{-1}f(t/a)$ $(t\in{\mathbb R}^n, a>0)$.
\end{abstract}

\smallskip\noindent
\small{\bf Keywords:\,}{\it wavelets; multiresolution analysis; optimisation; Douglas--Rachford algorithm; projection algorithm; feasibility problem.
}

\smallskip\noindent
\small{\bf AMS subject classifications}: 42C40, 42B99, 65T60, 47N10, 65K10, 65T60

\vskip0.2in
\centerline{\bf Dedication}
\vskip0.1in\noindent
{This paper is dedicated to the memory of Laureate Professor Jon Borwein, who first suggested this approach to the multidimensional wavelet construction problem. Jon was a friend and mentor to generations of mathematicians across the globe and has left an incomparable legacy of work spanning multiple disciplines. He was generous with his time and his ideas and was a highly respected and well-loved faculty member at the University of Newcastle in Australia.}



\setcounter{page}{1}

\section{Introduction}
\subsection{A brief history of wavelets}
Continuous wavelet decompositions have been used in analysis since the 1930's and in applied mathematics since the 1980's. They are implicit in the work of Calder\'on on singular integrals \cite{Cald} and explicit in  the work of Grossman and Morlet on seismic exploration \cite{GM}. They may be thought of as frame decompositions in which the index set associated with the frame is the upper half plane ${\mathbb R}^2_+=\{(x,a)\in{\mathbb R}^2:\, a>0\}$, and the frame elements are generated from a single {\it window} function $\psi$ by the action of dilations and translations. More precisely, given $f\in L^2({\mathbb R})$, we compute the frame coefficients $W_\psi f(x,a)$ by 
\begin{equation}
W_\psi f(x,a)=\langle f,\psi_{x,a}\rangle =\int_{-\infty}^\infty f(t)\dfrac{1}{\sqrt{a}}\overline{\psi}\left(\frac{x-t}a\right)\, dt .\label{CWT}
\end{equation}
The mapping $f\mapsto W_\psi f$ is known as the {\it continuous wavelet transform} (with respect to the wavelet $\psi$).
Given weak conditions on $\psi$, $f$ may be recovered from the frame coefficients $W_\psi f(x,a)$ $(x\in{\mathbb R},\ a>0)$(\cite{Daub2}). 

For applications, discretisations of the continuous transforms are desirable, so a theory of discrete wavelet frames (i.e., frames generated by the action of a discrete collection of dilations and translations of a single function $\psi$) was developed \cite{DS}, \cite{Hogan-Lakey} -- see also \cite{Gilbert et al} for connections with the theory of singular integrals. Unfortunately, these constructions failed to generalise to discrete data in such a way as to provide fast algorithms. On the other hand, Mallat \cite{Mall} and Meyer \cite{Mey1} independently developed the concept of {\it multiresolution analysis} (MRA) which enabled  fast algorithms. 

Realisations of MRA's require the construction of a {\it scaling function} $\varphi\in L^2({\mathbb R})$ with very special properties. As a minimum, it is necessary that $\varphi$ satisfies
   \begin{enumerate}
   \item[(i)] $\{\varphi (\cdot -k)\}_{k=-\infty}^\infty$ is an orthonormal collection in $L^2({\mathbb R})$.
   \item[(ii)] $\varphi$ is {\it self-similar} in the sense that there exists a sequence $\{h_k\}_{k=-\infty}^\infty\in\ell^2({\mathbb Z})$ such that 
   $$\dfrac{1}{2}\varphi\left(\dfrac{x}{2}\right)=\sum_{k=-\infty}^\infty h_k\varphi (x-k).$$
   \item[(iii)] $\int_{-\infty}^\infty\varphi (t)\, dt=1$.
   \end{enumerate}
Prototypical examples satisfying these conditions have long been known. The function $\varphi_H=\chi_{[0,1]}$, the characteristic function of $[0,1]$, is one such example and is associated with the {\it Haar} multiresolution analysis. Another example is $\varphi_S(t)=\sin (\pi t)/(\pi t)$ which is associated with the {\it Shannon} multiresolution analysis. Unfortunately, neither of these examples are satisfactory for use in signal analysis and processing for reasons we outline below.

When computing wavelet coefficients from (\ref{CWT}), it is much preferred that the wavelet $\psi$ be compactly supported, since this allows integration to be performed over a compact set. In fact, the shorter the support, the more efficiently this computation can be performed. Since the function $\varphi_S$ is not compactly supported (and, in fact, has very weak decay) it is therefore unsuitable.

The integral (\ref{CWT}) represents time localised information about the signal $f$ at scale $a$. With an application of the Parseval theorem for the Fourier transform, we have
\begin{equation}
W_\psi f(x,a)=\int_{-\infty}^\infty\hat f(\xi )e^{-2\pi ix\xi}\sqrt{a}\,\overline{\hat\psi}(-a\xi)\, d\xi \label{FT WT}
\end{equation}
(where $\hat f$ and $\hat\psi$ are the Fourier transforms of  $f$ and $\psi$ respectively). From (\ref{FT WT}) we see that the wavelet coefficients also give frequency localised information about $\hat f$ at the scale $a^{-1}$. For this reason it is desirable that $\hat\psi$ also be compactly supported. Of course $\psi$ and $\hat\psi$ cannot both be compactly supported, so we instead insist that $\hat\psi$ decay as fast as possible, or equivalently, that $\psi$ be as smooth as possible. Hence, for the purpose of efficient numerics, we shall add the following requirements to the three conditions above:
\begin{enumerate}
\item[(iv)] $\varphi$ is compactly supported.
\item[(v)] $\varphi$ is smooth.
\end{enumerate}
Note that the function $\varphi_H$ associated with the Haar multiresolution analysis fails condition (v), while the function $\varphi_S$ fails condition (iv). Without these properties, a multiresolution analysis fails to provide useful data and, in particular, without property (iv) a multiresolution analysis will not provide fast algorithms for discrete data.

Shortly after the publication of \cite{Mall} and \cite{Mey1}, Daubechies \cite{Daub1} used the MRA concept to construct a family of real-valued functions  ${}_N\varphi$ which satisfy conditions (i)--(v) and for which increasing the support (indexed by the positive integer $N$) gives improved smoothness. This led to constructions of scaling functions $\varphi$ with extra properties such as near-symmetry \cite{Daub2}.

Compactly supported wavelets with prescribed smoothness on ${\mathbb R}^n$ can be easily generated through tensor products of one-dimensional wavelets. However, such ``separable'' constructions suffer from the preferential treatment of the directions associated with the coordinate axes, and produce spurious artefacts in applications. Higher dimensional non-separable constructions have proved elusive when one uses the obvious generalisation of the dilations suggested by the one-dimensional approach. On a more fundamental level, the one-dimensional constructions cannot be easily transferred to higher dimensions as they involve techniques from complex analysis such as spectral factorisations which are not available in multivariate complex analysis. Indeed, Kova\v{c}evi\'c and Vetterli \cite{Kova-Vett} and Cohen and Daubechies \cite{Cohen-Daub} set out the theory of non-separable wavelets but did not explicitly construct any examples.
Ayache \cite{Ayache} and Belogay and Wang  \cite{Belogay} independently discovered methods of creating non-separable orthogonal wavelets in 1999. They were shortly followed by Lai and Roach \cite{Lai99}, He and Lai \cite{He00} and Karoui \cite{Karoui03,Karoui05}. San Antolin and Zalik \cite{SanAntolin13} discovered a family of non-separable scaling functions and their associated framelets by making a change of variables in specific trigonometric polynomials. All of these methods generate non-separable wavelets from one dimensional wavelets, typically by some kind of perturbation or modulation. For a more detailed discussion of the methods used, we refer the reader to Lai \cite{Lai02}. 

\subsection{This paper}
Here we employ techniques from optimisation to construct new MRA-based one-dimensional wavelets and new genuinely non-separable MRA-based multi-dimensional wavelets. We formulate the design problem in terms of constraints on a matrix-valued function well-known to wavelet theorists, discretise the problem, and then numerically compute -- through use of the {\it Douglas--Rachford algorithm} -- examples which simultaneously satisfy all of the constraints. This work is an extension of the PhD thesis of David Franklin \cite{Franklin}. A preliminary version of these results appears in \cite{FHT}.

This paper is organised as follows. In Section~\ref{sec: MRA} we review the basic axioms of a multiresolution analysis of $L^2({\mathbb R}^n)$ including details on how to encode properties of a scaling function $\varphi$ into an associated QMF $m_0$. These properties include the orthogonality of the integer shifts of $\varphi$, and the compact support and regularity of $\varphi$. In Section~\ref{sec: formulation}, we consider the relevant constraints on $m_0$ and the associated conjugate filters and express them  in terms of constraints on a matrix-valued function $U$ which has these filters as entries. We show that, in the case of compactly supported scaling functions and wavelets, sampling can be used to discretise the constraints. In Section~\ref{sec: opt prelim}, the relevant background material in optimisation and the Douglas--Rachford algorithm for solution of feasibility problems is introduced. This section provides a complete description of the relevant Hilbert spaces, constraints and projections for the wavelet construction problem. 
Finally, Section~\ref{sec: comp} includes computational results of the application of the Douglas--Rachford algorithm to the one-dimensional and two-dimensional wavelet construction problems.
 
 \subsection{Notation}
 We consider multi-indices $\alpha =(\alpha_1,\alpha_2,\dots ,\alpha_n)\in{\mathbb Z}_+^n$, (i.e., each $\alpha_i$ is a non-negative integer) and declare $|\alpha |=\sum_{j=1}^n\alpha_j$. The partial order on multi-indices is defined by $\beta\leq\alpha$ if and only if $\beta_j\leq\alpha_j$ for $1\leq j\leq n$. By $\partial^\alpha$ we mean the differential operator $$\partial^\alpha =\left(\dfrac{\partial}{\partial x_1}\right)^{\alpha_1}\cdots\left(\dfrac{\partial}{\partial x_n}\right)^{\alpha_n}.$$
 The collection of $N\times N$ matrices with complex coefficients is denoted ${\mathbb C}^{N\times N}$ and the sub-collection of unitary matrices by ${\mathcal U}(N)$. The Frobenius norm of an $N\times N$ matrix $A=(a_{ij})_{i,j=1}^N$ is given by $\|A\|_2=\left(\sum_{i,j=1}^N|a_{ij}|^2\right)^{1/2}$.
 
 Given positive integers $M$ and $n$, we define the set
 $$Q_M^n=\{0,1,\dots ,M-1\}^n=\{(j_1,j_2,\dots ,j_n)\in{\mathbb Z}^n;\, 0\leq j_i\leq M-1\text{ for }1\leq i\leq n\}.$$
By $({\mathbb C}^{N\times N})^{Q_M^n}$ we mean the collection of functions $F:Q_M^n\to{\mathbb C}^{N\times N}$.
 Elements of $({\mathbb C}^{N\times N})^{Q_M^n}$ are known as {\it matrix ensembles}.
 
 The dot product of $x$, $\xi\in{\mathbb R}^n$ is the real number $\langle x,\xi\rangle=\sum_{j=1}^nx_j\xi_j$ and we extend the dot product to $z$, $\zeta\in{\mathbb C}^n$ in the obvious way: $\langle z,\zeta\rangle =\sum_{j=1}^nz_j\zeta_j\in{\mathbb C}$. 
 
 The Fourier transform $\hat f$ of $f\in L^1({\mathbb R}^n)$ is normalised by $\hat f(\xi )=\int_{{\mathbb R}^n}f(x)e^{-2\pi i\langle x,\xi\rangle}\, dx$ and extends unitarily to $L^2({\mathbb R}^n)$.
 
A function $f:{\mathbb R}^n\to{\mathbb C}$ is said to be ${\mathbb Z}^n$-periodic if $f(\xi +\ell )=f(\xi )$ for all $\xi\in{\mathbb R}^n$ and $\ell\in{\mathbb Z}^n$.

The Lebesgue measure of a measurable subset $E\subset{\mathbb R}^n$ is denoted $|E|$.

\section{Multiresolution analysis, scaling functions and wavelets}\label{sec: MRA}

The construction of a compactly supported smooth orthogonal scaling function--wavelet pair $(\varphi ,\psi )$ on the line was first achieved by Daubechies in \cite{Daub1} with the help of the multiresolution structure introduced independently by Mallat \cite{Mall} and Meyer \cite{Mey1}. The problem reduces to the construction of a periodic matrix-valued function $U:{\mathbb R}\to{\mathbb C}^{2\times 2}$ satisfying certain restrictions designed to force $\varphi$ and $\psi$ to have desirable properties for signal processing. The $n$-dimensional wavelet construction problem may be reduced to the construction of a periodic matrix-valued function $U:{\mathbb R}^n\to{\mathbb C}^{2^n\times 2^n}$ satisfying similarly motivated restrictions. The construction relies on the notion of multiresolution analysis.  In this section, we give an explanation of the multiresolution structure and a discussion of the conditions we impose on the relevant filters to achieve these desirable properties.

\subsection{Multidimensional wavelets}

On $L^2({\mathbb R}^n)$ we have the unitary translation operators $\tau_x$ $(x\in{\mathbb R}^n)$ given by $\tau_xf(t)=f(t-x)$. Let $S$ be an  $n\times n$ matrix with integer entries, all of whose eigenvalues have absolute value greater than $1$, and define an associated dilation operator $D_S$ on $L^2({\mathbb R}^n)$ by $D_Sf(t)=(\det (S))^{-1/2}f(S^{-1}t)$. There are of course many possibilities for the matrix $S$ including (in two dimensions) the quincunx matrix $S=\left(\begin{matrix}1&1\\1&-1\end{matrix}\right)$. In this paper we consider only the matrices $S=2I_n$ (where $I_n$ is the $n\times n$ identity matrix) and in this case (with abusive notation) we write $D_2=D_{2I_n}$.

\subsection{Multiresolution analysis for $L^2({\mathbb R}^n)$}
A {\it multiresolution analysis} $(\{V_j\}_{j=\infty}^\infty ,\varphi )$ for $L^2({\mathbb R}^n)$ is a sequence of closed subspaces $\{V_j\}_{j=-\infty}^\infty\subset L^2({\mathbb R}^n)$ and a function $\varphi\in V_0$ such that 
\begin{enumerate}
\item[(i)] $V_j\subset V_{j+1}$ for all $j\in{\mathbb Z}$
\item[(ii)] $\cap_{j=-\infty}^\infty V_j=\{0\}$ and $\overline{\cup_{j=-\infty}^\infty V_j}=L^2({\mathbb R}^n)$
\item[(iii)] $f\in V_j\iff D_2^{-1}f\in V_{j+1}$
\item[(iv)] $f\in V_0\iff\tau_kf\in V_0$ $(k\in{\mathbb Z}^n)$
\item[(v)] $\{\tau_k\varphi\}_{k\in{\mathbb Z}^n}$ is an orthonormal basis for $V_0$.
\end{enumerate}

\subsubsection{Orthogonality}\label{sec: MRA orthog}
Orthonormality of the collection $\{\tau_k\varphi\}_{k\in{\mathbb Z}^n}$ is equivalent to the condition 
$$\sum_{k\in{\mathbb Z}^n} |\hat\varphi (\xi +k)|^2=1$$
for almost every $\xi$. Given such a collection, we note that $D_2\varphi\in V_{-1}\subset V_0$ and since $\{\tau_k\varphi\}_{k\in{\mathbb Z}^n}$ is an orthonormal basis for $V_0$, there exist constants $\{g_k^0\}_{k\in{\mathbb Z}^n}\in\ell^2({\mathbb Z}^n)$ such that 
\begin{equation}
\frac{1}{2^n}\varphi\left(\frac{x}{2}\right)=\sum_{k\in{\mathbb Z}^n} g_k^0\varphi (x-k).\label{dilation eqn HD}
\end{equation}
In fact, we have $g^0_k=2^{-n}\int_{{\mathbb R}^n}\varphi\left(\dfrac{x}{2}\right)\overline{\varphi (x-k)}\, dx$.  Taking the Fourier transform of both sides of (\ref{dilation eqn HD}) gives
\begin{equation}
\hat\varphi (2\xi )=m_0(\xi )\hat\varphi (\xi )\label{FT dilation eqn HD}
\end{equation}
where $m_0$ is the ${\mathbb Z}^n$-periodic Fourier series of $\{g^0_k\}$, i.e., $m_0(\xi )=\sum_{k\in{\mathbb Z}^n} g^0_ke^{-2\pi i\langle k,\xi\rangle}$ $(\xi\in{\mathbb R}^n)$. 

Let $V^n$ be the vertices of the unit cube $[0,1]^n$ in ${\mathbb R}^n$. Then $|V^n|=2^n$ and if $j\in\{0,1,\dots ,2^n-1\}$ has binary expansion $j=\sum_{k=0}^{n-1}a_k2^k$ $(a_k\in\{0,1\})$, we let $v_j=(a_0,a_1,\dots ,a_{n-1})\in V^n$. This provides a suitable enumeration of the elements of $V^n$, i.e., $V^n=\{v_j\}_{j=0}^{2^n-1}$. Note that $V^1=\{0,1\}\subset{\mathbb R}$ and 
$$V^2=\{v_0=(0,0),v_1=(1,0),v_2=(0,1),v_3=(1,1)\}\subset{\mathbb R}^2.$$
A necessary (but not sufficient) condition for the orthonormality of the collection $\{\tau_k\varphi\}_{k\in{\mathbb Z}^n}$ is the {\it quadrature mirror filter} (QMF) condition
\begin{equation}
\sum_{j=0}^{2^n-1}|m_0(\xi +v_j/2)|^2=1\label{QMF HD}
\end{equation}
for almost every $\xi$. 

Since $\det (2I_n)=2^n$, the index of the subgroup ${\mathbb Z}^n/2$ in ${\mathbb Z}^n$ is $2^n$. Attached to each of the $2^{n}-1$ non-trivial cosets $X_\varepsilon$ of ${\mathbb Z}^n/2$ in ${\mathbb Z}^n$ $(1\leq\varepsilon\leq 2^n-1)$ is a subspace $W_0^\varepsilon$ and a wavelet function $\psi^\varepsilon\in W_0^\varepsilon$ such that $V_1$ has the orthogonal decomposition 
\begin{equation}
V_1=V_0\oplus W_0^1\oplus W_0^2\oplus\cdots\oplus W_0^{2^n-1}.\label{subspace decomp}
\end{equation}
With $W_j^\varepsilon =D_{2^j}W_0^\varepsilon$ we  then have $L^2({\mathbb R}^n)=\oplus_{j=-\infty}^\infty (\oplus_{\varepsilon =1}^{2^n-1}W_j^\varepsilon )$ and the collection 
$$\{2^{j/2}\psi^\varepsilon(2^jx-k):\, j\in{\mathbb Z},\ k\in{\mathbb Z}^n,\ 1\leq\varepsilon\leq 2^n-1\}$$ 
forms an orthonormal basis for $L^2({\mathbb R}^n)$.

Since $D_2\psi^\varepsilon\in W^\varepsilon_{-1}\subset V_0$, there are constants $g_k^\varepsilon$ such that 
\begin{equation}
\frac{1}{2^n}\psi^\varepsilon\left(\frac{x}{2}\right)=\sum_{k\in{\mathbb Z}^n} g_k^\varepsilon\varphi (x-k)\quad (1\leq\varepsilon\leq 2^n-1).\label{wavelet dilation eqn HD}
\end{equation}
The Fourier transform of (\ref{wavelet dilation eqn HD}) may be written as
$\widehat{\psi^\varepsilon} (2\xi )=m_\varepsilon (\xi )\hat\varphi (\xi )$, 
where $m_\varepsilon$ is the ${\mathbb Z}^n$-periodic Fourier series of $\{g_k^\varepsilon\}$, i.e., $m_\varepsilon(\xi )=\sum_{k\in{\mathbb Z}^n} g_k^\varepsilon e^{-2\pi i\langle k,\xi\rangle}$ $(\xi\in{\mathbb R}^n)$. Given the orthonormality of $\{\tau_k\varphi\}_{k\in{\mathbb Z}^n}$, the orthonormality of $\{\tau_k\psi^\varepsilon\}_{k\in{\mathbb Z}^n}$ becomes equivalent to 
\begin{equation}
\sum_{j=0}^{2^n-1}|m_\varepsilon (\xi +v_j/2)|^2=1\label{QMF2 HD}
\end{equation}
for almost every $\xi$. Furthermore, the orthogonality of the decomposition (\ref{subspace decomp})  requires
\begin{equation}
\sum_{j=0}^{2^n-1}m_\varepsilon (\xi +v_j/2)\overline{m_\eta (\xi +v_j/2)}=\delta_{\varepsilon\eta}\label{cross QMF HD}
\end{equation}
for almost every $\xi$.

\subsubsection{Compact support}\label{cs HD}

The requirement $\overline{\cup_{j=-\infty}^\infty V_j}=L^2({\mathbb R}^n)$ of a multiresolution analysis forces $|\hat\varphi (0)|=|\int_{{\mathbb R}^n} \varphi (t)\, dt|=1$. It is convenient to choose the phase of $\varphi$ so that $\hat\varphi (0)=1$. Iterating equation (\ref{FT dilation eqn HD}) gives 
$$\hat\varphi (\xi )=m_0(\xi /2)m_0(\xi /4)\hat\varphi (\xi /4)=\cdots =\prod_{j=1}^Jm_0(\xi /2^j)\hat\varphi (\xi /2^J).$$
If $m_0$ satisfies the QMF condition (\ref{QMF HD}) and the infinite product $\prod_{j=1}^\infty m_0(\xi /2^j)$ converges pointwise almost everywhere, then its limit $\hat\varphi$ is square integrable and $\|\varphi\|_2=1$ \cite{Daub2}. 

It is relatively easy to see that if $\varphi$ is supported on $[0,M-1]^n\subset{\mathbb R}^n$, then the coefficients $h_k=h_{k_1k_2\dots k_n}$ in the dilation equation (\ref{dilation eqn HD}) are zero unless $0\leq k_i\leq M-1$ $(1\leq i\leq n)$. The converse is trickier, and requires a higher-dimensional version of the Paley--Wiener theorem (see Theorem \ref{PW HD} below).

A function  $F:D\subset{\mathbb C}^n\to{\mathbb C}$ is {\it holomorphic} on $D$  if for each $z^0=(z_1^0,z_2^0,\dots ,z_n^0)\in D$, there is a polydisc
$$P=\{(z_1,z_2,\dots ,z_n)\in{\mathbb C}^n:\, |z_1-z_1^0|<r_1,\ \dots ,|z_n-z_n^0|<r_n\}\subset D$$
$(r_1,\dots ,r_n>0)$ in which $F$ may be represented by the absolutely  convergent series
$$F(z)=F(z_1,z_2,\dots ,z_n)=\sum_{k_1,k_2,\dots ,k_n\geq 0} a_{k_1k_2\dots k_n}(z_1-z_1^0)^{k_1}(z_2-z_2^0)^{k_2}\cdots (z_n-z_n^0)^{k_n}.$$
We say $F$ is {\it entire} if it is holomorphic on $D={\mathbb C}^n$. An entire function $F:{\mathbb C}^n\to{\mathbb C}$ is of {\it exponential type} $R>0$ if for each $\varepsilon >0$ there is a constant $A_\varepsilon >0$ such that 
$$|F(z)|\leq A_\varepsilon e^{2\pi (R+\varepsilon )\|z\|_1}$$
where if $z=(z_1,z_2,\dots ,z_n)$, $\|z\|_1=\sum_{j=1}^n|z_j|$. The class of all functions of exponential type $R>0$ on ${\mathbb C}^n$ is denoted ${\mathcal E}^n(R)$.

Suppose $F$ is the inverse Fourier transform of a function $\sigma\in L^2({\mathbb R}^n)$ which vanishes outside 
$$[-R,R]^n=\bigg\{\xi\in{\mathbb R}^n:\, \|\xi\|_\infty=\max_{1\leq j\leq n}|\xi_j|\leq R\bigg\},$$
i.e., $F(z)=\int_{[-R,R]^n}\sigma (\xi )e^{2\pi i\langle z,\xi\rangle}\, d\xi$ where, if $z=x+iy\in {\mathbb C}^n$ $(x,y\in{\mathbb R}^n)$ and $\xi\in{\mathbb R}^n$, we have $\langle z,\xi\rangle=\sum_{j=1}^nz_i\xi_j=\langle x,\xi\rangle +i\langle y,\xi\rangle$. Then $F$ satisfies the pointwise bound
$$|F(z)|\leq\int_{[-R,R]^n}|\sigma (\xi )|e^{-2\pi\langle y,\xi\rangle}\, d\xi .$$
However, if $\xi\in [-R,R]^n$, then $|\xi_j|\leq R$ for each $1\leq j\leq n$ so that $e^{-2\pi y_j\xi_j}\leq e^{2\pi R|y|}$ and as a consequence
$$|F(z)|\leq e^{2\pi R\|y\|_1}\int_{[-R,R]^n}|\sigma (\xi )|\, d\xi\leq 2^{n/2}R^{n/2}\|\sigma\|_2e^{2\pi R\|z\|_1}.$$
Hence, $F\in{\mathcal E}^n(R)$.

The following multidimensional generalisation of the Paley--Wiener theorem is a special case of a result given by Stein and Weiss \cite{SW} for more general support sets.

\begin{theorem}[Paley--Wiener theorem for cubes]\label{PW HD} Suppose $F\in L^2({\mathbb R}^n)$. Then $F$ is the inverse Fourier transform of a function vanishing outside the cube $[-R,R]^n$ if and only if $F$ is the restriction to ${\mathbb R}^n$ of a function in ${\mathcal E}^n(R)$.
\end{theorem}

Given a positive integer $N$, we say $\Gamma :{\mathbb C}^n\to{\mathbb C}$ is a trigonometric polynomial of degree $N$ if $\Gamma (\zeta )=\sum_{k_1,k_2,\dots ,k_n=0}^Na_{k_1k_2\dots k_n}e^{-2\pi i\langle k,\zeta\rangle}$ $(\zeta\in{\mathbb C}^n)$ for some $\{a_{k_1k_2\dots k_n}\}_{k_1,k_2,\dots ,k_n=0}^N\subset{\mathbb C}$.

The following result is a multi-dimensional version of Lemma 6.2.2 of \cite{Daub2}.

\begin{proposition} Suppose $\Gamma$ is a trigonometric polynomial of degree $N$ on ${\mathbb C}^n$ and $\Gamma (0)=1$. Let 
$$F(\zeta )=\prod_{j=1}^\infty\Gamma (\zeta /2^j)\qquad (\zeta\in{\mathbb C}^n).$$ 
Then $F$ is the Fourier transform of a function $f\in L^2({\mathbb R}^n)$ supported on the cube $[0,N]^n$.
\end{proposition}

\begin{proof} We prove the result in the case $n=2$ only.  Let $\Gamma (\zeta )=\sum_{k_1,k_2=0}^Na_{k_1k_2}e^{-2\pi i\langle k,\zeta\rangle}$ be as in the statement of the proposition. Note that 
$|e^{-2\pi i\langle k,\zeta\rangle}-1|\leq2\pi\|k\|_\infty\|\zeta\|_1$
so that 
\begin{align}
|\Gamma (\zeta )|\leq 1+|\Gamma (\zeta )-1|&\leq 1+\bigg|\sum_{k_1,k_2=0}^Na_{k_1k_2}(e^{-2\pi i\langle k,\zeta\rangle}-1)\bigg|\notag\\
&\leq 1+\sum_{k_1,k_2=0}^N|a_{k_1,k_2}||e^{-2\pi i\langle k,\zeta\rangle}-1|\notag\\
&\leq 1+\sum_{k_1,k_2=0}^N2\pi\|k\|_\infty \|\zeta\|_1\leq 1+C\|\zeta\|_1\leq e^{C\|\zeta\|_1}\label{exp est}
\end{align}
where $C=2\pi N\sum_{k_1,k_2=0}^N|a_{k_1,k_2}|$.
However, if $\|\zeta\|_1\leq 1$, from (\ref{exp est}) we have
$\bigg|\prod_{j=1}^\infty\Gamma (\zeta /2^j)\bigg|\leq\prod_{j=1}^\infty e^{C\|\zeta\|_1/2^j}=e^{C\|\zeta\|_1}\leq e^C$,
while if $\|\zeta\|_1>1$, 
\begin{align*}
|e^{-2\pi i\langle k,\zeta\rangle}-1|&=|e^{-2\pi i\langle k,x\rangle}e^{2\pi\langle k,y\rangle}-e^{2\pi\langle k,y\rangle}+e^{2\pi\langle k,y\rangle}-1|\\
&\leq e^{2\pi\langle k,y\rangle}|e^{-2\pi i\langle k,x\rangle}-1|+|e^{2\pi\langle k,y\rangle}-1|\leq 3
\end{align*}
provided $y_1,y_2\leq 0$. Here $\zeta =x+iy$ with $x=(x_1,x_2)$, $y=(y_1,y_2)\in{\mathbb R}^2$. Therefore, if $\|\zeta\|_1>1$ and $y_1,y_2\leq 0$, we have
\begin{align*}
|\Gamma (\zeta )|&\leq 1+|\Gamma (\zeta )-1|\\
&\leq 1+\sum_{k_1,k_2=0}^N|a_{k_1,k_2}||e^{-2\pi i\langle k,\zeta\rangle}-1|\leq 1+3\sum_{k_1,k_2=0}^N|a_{k_1,k_2}|=C.
\end{align*}
We choose an integer $j_0\geq 0$ such that $2^{j_0}\leq\|\zeta\|_1<2^{j_0+1}$. Then 
$$\bigg|\prod_{j=1}^\infty\Gamma(\zeta /2^j)\bigg|\leq\prod_{j=1}^{j_0} C\prod_{j=j_0+1}^\infty e^{C\|\zeta\|_1/2^j}=C^{j_0}\exp (C\|\zeta\|_1/2^{j_0})\leq C\|\zeta\|_1^\alpha$$
with $\alpha =\dfrac{\ln C}{\ln 2}$. We conclude that if $y_1,y_2\leq 0$,
\begin{equation}
\bigg|\prod_{j=1}^\infty\Gamma (\zeta /2^j)\bigg|\leq C\max\{1,\|\zeta\|_1^\alpha\}.\label{prod est 1}
\end{equation}
Suppose now that $y_1>0$, $y_2\leq 0$. Then 
$\Gamma (\zeta )=e^{-2\pi iN\zeta_1}\tilde\Gamma (\zeta)$ with 
$$\tilde\Gamma (\zeta )=\sum_{k_1,k_2=0}^Nb_{k_1,k_2}e^{-2\pi ik_1(-\zeta_1)}e^{-2\pi ik_2\zeta_2}$$
and $b_{k_1,k_2}=a_{N-k_1,k_2}$ so that
\begin{equation}\bigg|\prod_{j=1}^\infty\Gamma (\zeta /2^j)\bigg|=|e^{-2\pi iN\zeta_1}|\bigg|\prod_{j=1}^\infty |\tilde\Gamma (\zeta /2^j)\bigg|\leq Ce^{2\pi Ny_1}\max\{1,\|\zeta\|_1^\alpha\}.\label{prod est 2}
\end{equation}
Similarly, if $y_1\leq 0$, $y_2\geq 0$,
\begin{equation}
\bigg|\prod_{j=1}^\infty\Gamma (\zeta /2^j)\bigg|\leq Ce^{2\pi Ny_2}\max\{1,\|\zeta\|_1^\alpha\}\label{prod est 3}
\end{equation}
and if $y_1,y_2\geq 0$,
\begin{equation}
\bigg|\prod_{j=1}^\infty\Gamma (\zeta /2^j)\bigg|\leq Ce^{2\pi N(y_1+y_2)}\max\{1,\|\zeta\|_1^\alpha\}.\label{prod est 4}
\end{equation}
Let $P(\zeta )=e^{\pi iN(\zeta_1+\zeta_2)}\Gamma (\zeta )$.  Since $|e^{\pi iN(\zeta_1+\zeta_2)}|=e^{-\pi N(y_1+y_2)}$ we have
\begin{equation}
\bigg|\prod_{j=1}^\infty P(\zeta /2^j)\bigg|=e^{-\pi N(y_1+y_2)}\prod_{j=1}^\infty |\Gamma (\zeta /2^j)|.\label{P and Gamma}
\end{equation}
Applying (\ref{P and Gamma}) to (\ref{prod est 1})--(\ref{prod est 4}) gives
\begin{equation*}
\bigg|\prod_{j=1}^\infty P(\zeta /2^j)\bigg|\leq Ce^{\pi N\|y\|_1}\max\{1,\|\zeta\|_1^\alpha\}\leq C_\varepsilon e^{\pi (N+\varepsilon )\|\zeta\|_1}
\end{equation*}
for all $\zeta\in{\mathbb C}^n$. By Theorem \ref{PW HD}, the product $\prod_{j=1}^\infty P(\zeta /2^j)$ is the Fourier transform of a function $\sigma\in L^2({\mathbb R}^n)$ supported on the cube $[-N/2,N/2]^n$. But
$$F(\zeta )=\prod_{j=1}^\infty\Gamma (\zeta /2^j)=e^{-\pi iN(\zeta_1+\zeta_2)}\prod_{j=1}^\infty P(\zeta /2^j)=e^{-\pi iN(\zeta _1+\zeta_2)}\hat\sigma (\zeta ).$$
If $f(x)=\sigma (x-(N/2,N/2))$, then $f$ is supported on $[0,N]^2$ and $\hat f(\zeta )=e^{-\pi iN(\zeta_1+\zeta_2)}\hat\sigma (\zeta )=F(\zeta )$.
\end{proof}

\subsubsection{Completeness}\label{sec: completeness}
Since $m_0$ satisfies (\ref{QMF HD}), we have $|m_0(\xi )|\leq 1$ for almost every $\xi$. Since we require $\hat\varphi (0)=1$, equation (\ref{FT dilation eqn HD}) requires $m_0(1)=1$. Because of the MRA condition (\ref{QMF HD}), we also have $m_0(v_j/2)=0$ for $1\leq j\leq 2^n-1$. The cross QMF condition (\ref{cross QMF HD}) now gives $m_\varepsilon (0)=0$ for $1\leq\varepsilon\leq 2^n-1$. Summarising, we have
\begin{equation}
m_0 (v_j/2)=\delta_{j0},\quad m_\varepsilon (0)=0\quad (0\leq j\leq 2^n-1,\ 1\leq\varepsilon\leq 2^n-1).\label{completeness condits}
\end{equation}

\subsubsection{Regularity}\label{ssec: regularity}

The following result is a consequence of \cite[Chapter 3.7, Proposition 4]{Meyer89}.

\begin{theorem}\label{thm: regularity} 
Suppose  $\varphi$ is a compactly supported scaling function and  $\{\psi^\varepsilon\}_{\varepsilon =1}^{2^n-1}$ is a collection of wavelets associated with an MRA of $L^2({\mathbb R}^n)$, all of which have bounded partial derivatives of order less than or equal to $d$.  Then 
\begin{enumerate}
\item[(i)] $\int_{{\mathbb R}^n} x^\alpha\psi^\varepsilon (x)\, dx=0$ for $|\alpha |\leq d$ and $1\leq\varepsilon\leq 2^n-1$.
\item
[(ii)] The conjugate filters $\{m_\varepsilon\}_{\varepsilon =1}^{2^n-1}$ satisfy
\begin{equation}
\partial^\alpha m_\varepsilon (\xi )\bigg|_{\xi =0}=0\quad\text{ for $1\leq\varepsilon\leq 2^n-1$, $|\alpha |\leq d$.}\label{flat filters}
\end{equation}
\end{enumerate}
\end{theorem}
Condition (ii) is not sufficient to ensure regularity of the  wavelets $\{\psi^\varepsilon\}_{\varepsilon =1}^{2^n-1}$. Nevertheless, this is the condition we impose in an attempt to enforce regularity, with the expectation that the larger the value of $d$ (i.e., the ``flatter'' the filters  $m_\varepsilon$ $(1\leq\varepsilon\leq 2^n-1)$ at the origin) the higher the regularity.

Flatness of the coniugate filters at the origin as in (\ref{flat filters}) coupled with the cross QMF condition (\ref{cross QMF HD}) gives the flatness of the quadrature filter $m_0$ at the points $\{v_j/2\}_{j=1}^{2^n-1}$ as the next result shows.

\begin{proposition} \label{cor: flat QMF} Suppose $\{m_\varepsilon\}_{\varepsilon =0}^{2^n-1}$ are trigonometric polynomials satisfying (\ref{cross QMF HD}) and (\ref{flat filters}) and $m_0(0)=1$. Then $m_0$ satisfies
\begin{equation}
\partial^\alpha m_0(\xi )\bigg|_{\xi =v_j/2}=0\quad\text{ for $1\leq j\leq 2^n-1$, $|\alpha |\leq d$.}\label{m_0 flat}
\end{equation}
\end{proposition}

\subsubsection{Non-separability}\label{ssec: nonseparability}
It is a simple matter to construct smooth orthogonal compactly supported wavelets on ${\mathbb R}^n$ through a tensor-product construction. If $n=2$, we let $(\varphi_1 ,\psi_1 )$ and $(\varphi_2,\psi_2)$  be  smooth orthogonal compactly supported one-dimensional scaling function-wavelet pairs and define a two-dimensional scaling function $\Phi$ and three two-dimensional wavelets $\Psi_1,\Psi_2,\Psi_3$ by
\begin{align}
&\Phi (\xi _1,\xi_2)=\varphi_1 (\xi_1)\varphi_2 (\xi _2);\label{sep1}\\
\Psi_1(\xi _1,\xi_2)=\varphi_1(\xi_1)\psi_2(\xi_2);\quad&\Psi_2(\xi_1,\xi_2)=\psi_1(\xi_1)\varphi_2 (\xi_2);\quad \Psi_3(\xi_1,\xi_2)=\psi_1 (\xi_1)\psi_2 (\xi_2).\notag
\end{align}
Then the collection $\{\Psi_\varepsilon\}_{\varepsilon =1}^3$ generates a smooth orthogonal compactly supported wavelet basis on ${\mathbb R}^2$. Such systems, however, perform poorly in image processing applications, producing artefacts in the directions of the coordinate axes \cite{Kova-Vett}. Here we seek {\it non-separable} wavelet bases in which neither the scaling function nor the wavelets can be decomposed as the tensor product of two functions of a single variable. Although non-separability is not imposed as a constraint, it is a simple matter to check whether scaling functions and wavelets generated by our methods are separable.

Suppose a two-dimensional scaling function $\Phi$ is supported on $[0,M]^2$ and separable as in (\ref{sep1}). Let $m_0$ be the two-dimensional scaling filter associated with $\Phi$ and let $m_0^{(1)}$, $m_0^{(2)}$ be the one-dimensional scaling filters associated with $\varphi_1$ and $\varphi_2$ respectively. Then $m_0$ is separable:
$$m_0(\xi_1,\xi_2)=m_0^{(1)}(\xi_1)m_0^{(2)}(\xi_2)$$
and since $m_0^{(1)}(0)=m_0^{(2)}(0)=1$, we have
\begin{equation}
m_0(\xi_1,\xi_2)=m_0(\xi_1,0)m_0(0,\xi_2).\label{sep2}
\end{equation}
Recalling that $m_0(\xi_1,\xi_2)=\sum_{k_1,k_2=0}^{M-1}g_{k_1,k_2}^0e^{-2\pi i(k_1\xi_1+k_2\xi_2)}$, (\ref{sep2}) becomes 
$$\sum_{k_1,k_2=0}^{M-1}g_{k_1,k_2}^0e^{-2\pi i(k_1\xi_1+k_2\xi_2)}=\sum_{k_1,\ell=0}^{M-1}g_{k_1,\ell}^0e^{-2\pi ik_1\xi_1}\sum_{k_2,n=0}^{M-1}g_{n,k_2}^0e^{-2\pi ik_2\xi_2}$$
which is equivalent to
\begin{equation}
g_{p,q}^0=\bigg(\sum_{\ell=0}^{M-1}g_{p,\ell}^0\bigg)\bigg(\sum_{n=0}^{M-1}g_{n,q}^0\bigg).\label{sep3}
\end{equation}
Let $G^0$ be the $M\times M$ matrix with $(j,k)$-th entry $G^0_{j,k}=g^0_{j,k}$ $(0\leq j,k\leq M-1)$. Then (\ref{sep3}) is equivalent to the statement
$G^0=(G^0{\mathbf 1})((G^0)^T{\mathbf 1})^T$
where ${\mathbf 1}=(1,1,\dots ,1)^T\in{\mathbb R}^M$. As a measure of the separability of a two-dimensional scaling function $\varphi$, we compute its separability measure
$$S(\varphi )=\|G^0-(G^0{\mathbf 1})((G^0)^T{\mathbf 1})^T\|_2$$
where $\|\cdot\|_2$ is the Frobenius norm. Note that $S(\varphi )=0$ if and only if $\varphi$ is separable. We seek scaling functions with separability measure significantly larger than zero.

\section{Matrix formulation and  discretisation}\label{sec: formulation}

Equations (\ref{QMF HD}), (\ref{QMF2 HD}) and (\ref{cross QMF HD}) may be neatly organised as follows: the orthogonality of the collections $\{\tau_k\varphi\}_{k\in{\mathbb Z}^n}$ and $\{\tau_k\psi^\varepsilon\}_{k\in{\mathbb Z}^n}$ $(1\leq\varepsilon\leq 2^n-1)$ and the orthogonality of the spaces they span requires that the matrix-valued ${\mathbb Z}^n$-periodic function $U:{\mathbb R}^n\to{\mathbb C}^{2^n\times 2^n}$ given by 
\begin{equation}
U(\xi )_{j,\varepsilon }=m_\varepsilon (\xi +v_j/2)\qquad (0\leq j,\varepsilon\leq 2^n-1)\label{paraunitarity HD}
\end{equation}
is unitary for all $\xi$. When $n=1$, $U(\xi )$ is the $2\times 2$ matrix 
$$U(\xi )=\left(\begin{matrix}m_0(\xi )&m_1(\xi )\\
m_0(\xi +\frac{1}{2})&m_1(\xi +\frac{1}{2})\end{matrix}\right)$$
with $\xi\in{\mathbb R}$, while when $n=2$, $U(\xi )$ is the $4\times 4$ matrix
$$\left(\begin{matrix} m_0(\xi _1,\xi_2)&m_1(\xi_1,\xi_2)&m_2(\xi_1,\xi_2 )&m_3(\xi_1,\xi_2 )\\
m_0(\xi_1+\frac{1}{2},\xi_2 )&m_1(\xi_1+\frac{1}{2},\xi_2)&m_2(\xi_1+\frac{1}{2},\xi_2)&m_3(\xi_1+\frac{1}{2},\xi_2)\\
m_0(\xi_1,\xi_2+\frac{1}{2})&m_1(\xi_1,\xi_2+\frac{1}{2})&m_2(\xi_1,\xi_2+\frac{1}{2})&m_3(\xi_1,\xi_2+\frac{1}{2})\\
m_0(\xi+\frac{1}{2},\xi_2+\frac{1}{2})&m_1(\xi_1+\frac{1}{2},\xi_2+\frac{1}{2})&m_2(\xi_1+\frac{1}{2},\xi_2+\frac{1}{2})&m_3(\xi_1+\frac{1}{2},\xi_2+\frac{1}{2})\end{matrix}\right)$$
with $\xi=(\xi_1,\xi_2)\in{\mathbb R}^2$. 

The matrix-valued function $U$ of (\ref{paraunitarity HD}) 
holds the key to our approach to wavelet construction in one- and higher dimensions. In this section, we record the conditions on $U$ which encode the orthogonality, compact support and regularity conditions on the filters $m_\varepsilon$ of Section~\ref{sec: MRA orthog}. Then we explore a sampling-based approach to discretisation of the problem through use of the discrete Fourier transform. Finally, we use this discretisation to express the problem of wavelet construction as a feasibility problem in which the constraint sets live in a finite-dimensional Hilbert space of matrix ensembles.

\subsection{Matrix formulation}
In this section, the orthogonality, regularity and compact support conditions imposed on a scaling function and its associated wavelets are couched in terms of the matrix-valued function $U$ of (\ref{paraunitarity HD}). It is clear from the form of (\ref{paraunitarity HD}) that there are strong relationships between the rows of $U(\xi )$, and these relationships -- known here as {\it consistency conditions} --  also must be accounted for when designing such matrices for wavelet construction.

\subsubsection{Consistency}\label{sec: consistency HD}
Let $V^n$ be as in Section \ref{sec: MRA orthog}. We endow $V^n$ with a group structure, thinking of it as $({\mathbb Z}_2)^n$ with coordinate-wise addition modulo $2$:
\begin{equation}
(v_j\oplus v_k)_\ell=(v_j)_\ell +(v_k)_\ell\ \text{(mod $2$)}.\label{group op}
\end{equation}
Each $j\in Y_n=\{0,1,\dots ,2^n-1\}$ determines a permutation $\tau_j$ of $V_n$ given by 
$$v_{\tau_j(k)}=v_j\oplus v_k$$
and a permutation matrix $\sigma_j\in{\mathbb C}^{2^n\times 2^n}$ with $(k,\ell )$-th entry
\begin{equation}
(\sigma_j)_{k\ell}=\begin{cases}1&\text{ if $\tau_j(k)=\ell$}\\
0&\text{ else.}
\end{cases}\label{sigma def nD}
\end{equation}
Since $v_j\oplus v_k=v_k\oplus v_j$, $\sigma_j$ is symmetric. Further, since each $m_\varepsilon$ is ${\mathbb Z}^n$-periodic,
\begin{align*}
[\sigma_jU(\xi )]_{k\varepsilon}=\sum_{\ell =0}^{2^n-1}(\sigma_j)_{k\ell}U(\xi )_{\ell\varepsilon}&=\sum_{\{\ell ;\, \tau_j(k)=\ell\}}U(\xi )_{l\varepsilon}\\
&=U(\xi )_{\tau_j(k),\varepsilon}\\
&=m_\varepsilon (\xi +v_{\tau_j(k)}/2)\\
&=m_\varepsilon (\xi +(v_j\oplus v_k)/2)\\
&=m_\varepsilon (\xi +(v_j+v_k)/2)=U(\xi +v_j/2)_{k\varepsilon},
\end{align*}
from which we conclude that 
\begin{equation}
U(\xi +v_j/2)=\sigma_jU(\xi )\label{consistency 1}
\end{equation}
for all $\xi\in{\mathbb R}^n$ and all $v_j\in V^n$.

Since addition in $({\mathbb Z}_2)^n$ is commutative, so too is the collection of matrices $\{\sigma_j\}_{j\in Y_n}$.

\begin{proposition}\label{prop: perm comm} If $j\in Y_n$ has binary representation $j=\sum_{k=0}^{n-1}a_k2^k$ $(a_k\in\{0,1\})$ then $\sigma_j$ decomposes as 
$$\sigma_j=\prod_{k=0}^{n-1}(\sigma_{2^k})^{a_k}.$$
\end{proposition}

\begin{proof} Note that if $j$ has binary representation as in the statement of the proposition, then $v_j=(a_0, a_1,\dots ,a_{n-1})=\sum_{k=0}^{n-1}a_kv_{2^k}$. We apply (\ref{consistency 1}) repeatedly to find
\begin{align}
\sigma_jU(\xi )&=U(\xi +v_j/2)\notag\\
&=U\bigg(\xi +\sum_{k=0}^{n-1}a_kv_{2^k}/2\bigg)\notag\\
&=(\sigma_{2^{n-1}})^{a_{n-1}}U\bigg(\xi +\sum_{k=0}^{n-2}a_kv_{2^k}/2\bigg)
=\prod_{k=0}^{n-1}(\sigma_{2^k})^{a_k}U(\xi ).\label{prod perms}
\end{align}
Since the matrices $\{\sigma_j\}_{j\in Y_n}$ commute, the product in (\ref{prod perms}) is independent of the order of the factors, and the product is well-defined. Since $U(\xi )$ is unitary, the result follows from (\ref{prod perms}).
\end{proof}

\begin{corollary} The consistency condition (\ref{consistency 1}) holds for all $j\in Y_n$  and all $\xi\in{\mathbb R}^n$ if and only if 
$$U(\xi +v_{2^k}/2)=\sigma_{2^k}U(\xi )$$
for all $0\leq k\leq n-1$.
\end{corollary}

\subsubsection{Orthogonality/Unitarity}

The unitarity of the matrix $U(\xi )$ of (\ref{paraunitarity HD}) is not sufficient to ensure the orthogonalities we require.  In one dimension, {\it Cohen's condition} \cite{Coh} provides an easily checked sufficient condition. The following result (due to Bownik \cite{bownik}) is a generalisation of the one-dimensional Cohen condition. 

\begin{theorem}\label{thm: regularity-bownik} Suppose $m_0\in C^\infty ({\mathbb R}^n)$ is ${\mathbb Z}^n$-periodic and is such that the infinite product $\hat\varphi (\xi )=\prod_{j=1}^\infty m_0(2^{-j}\xi )$ converges in $L^2({\mathbb R}^n)$. Suppose also that there exists a compact set $K\subset{\mathbb R}^n$ such that 
\begin{enumerate}
\item[(i)] $K$ contains a neighbourhood of the origin;
\item[(ii)] $|K\cap (\ell +K)|=\delta_{\ell 0}$ for all $\ell\in{\mathbb Z}^d$;
\item[(iii)] $m_0(2^{-j}\xi )\neq 0$ for all integers $j> 0$ and all $\xi\in K$.
\end{enumerate}
Then $\{\varphi (\cdot -k)\}_{k\in{\mathbb Z}^n}$ forms an orthonormal set. If $m_0\in C^D({\mathbb R}^n)$ is ${\mathbb Z}^n$-periodic with $D>n/2$, then the converse is true.
\end{theorem}
Note that if $m_0$ is a trigonometric polynomial, then it is infinitely differentiable. From Theorem \ref{thm: regularity-bownik} we see that  if $m_\varepsilon $ $(0\leq\varepsilon\leq 2^n-1)$ are trigonometric polynomials for which the matrix $U(\xi )$ given by (\ref{paraunitarity HD}) is unitary, then the orthogonalities we require will be assured provided  $m_0$ has no zeroes on $[-1/4,1/4]^n$.

\subsubsection{Compact support}
In Section~\ref{cs HD}, we saw that  $\varphi$ being supported on $[0,M-1]^n$ is equivalent to the Fourier series $m_0$ being a trigonometric polynomial of the form $m_0(\xi )=\sum_{k\in Q_M^n}h_ke^{-2\pi i\langle k,\xi\rangle}$. The compact support of the wavelets $\psi_\varepsilon$ $(1\leq\varepsilon\leq 2^n-1)$ is equivalent to the Fourier series $m_\varepsilon$ $(1\leq\varepsilon\leq 2^n-1)$ having a similar form. This forces the matrix $U=U(\xi )$ to also be a trigonometric polynomial:
\begin{equation}
U(\xi )=\sum_{k\in Q_M^n}A_ke^{-2\pi i\langle k,\xi\rangle}\label{trig poly HD}
\end{equation}
where for each $k\in Q_M^n$,  $A_k$ is a constant $2^n\times 2^n$ matrix whose entries are the coefficients $g_k^\varepsilon$.

\subsubsection{Completeness}
As in section \ref{sec: completeness}, we note that the density of  $\cup_{j=-\infty}^\infty V_j$ in  $L^2({\mathbb R}^n)$ requires the conditions (\ref{completeness condits}) on the Fourier series $\{m_\varepsilon\}_{\varepsilon}^{2^n-1}$.  We define
$$1\otimes {\mathcal U}(2^n-1)=\left\{\left(\begin{matrix}1&{\mathbf 0}^T\\{\mathbf 0}&V\end{matrix}\right):\, V\in {\mathcal U}(2^n-1)\right\}.$$
Here ${\mathbf 0}=(0,0,\dots ,0)^T\in{\mathbb C}^n$.
Since $U(\xi )$ is unitary, the conditions (\ref{completeness condits}) can be summarised as:
\begin{equation*}
U(0)\in1\otimes {\mathcal U}(2^n-1).
\end{equation*}

\subsubsection{Regularity}
We define
$${\mathbb C}\otimes{{\mathbb C}^{(2^n-1)\times (2^n-1)}}=\left\{\left(\begin{matrix}b&{\mathbf 0}^T\\{\mathbf 0}&B\end{matrix}\right):\, b\in{\mathbb C},\ B\in{\mathbb C}^{(2^n-1)\times (2^n-1)}\right\}.$$
To enable the regularity of the wavelets we construct, we impose condition (\ref{flat filters}) of Theorem \ref{thm: regularity} on $\{m_\varepsilon\}_{\varepsilon =1}^{2^n-1}$. As we saw in Corollary \ref{cor: flat QMF}, this implies condition (\ref{m_0 flat}) on $m_0$. Together these conditions may be written in terms of the matrix-valued function $U$ of (\ref{paraunitarity HD}) as follows:
\begin{equation}
\partial^\alpha U(\xi )\bigg|_{\xi =0}\in{\mathbb C}\otimes{{\mathbb C}^{(2^n-1)\times (2^n-1)}}\text{ for $1\leq|\alpha |\leq d$.}\label{regularity U}
\end{equation}

In summary, the problem of the construction of compactly supported orthogonal smooth scaling functions and wavelets on the line is equivalent to the following:

\begin{problem}[Scaling function/wavelet pairs in ${\mathbb R}^n$]\label{prob:p1}
Given an even integer $M\geq 4$, we seek matrices $\{A_k\}_{k\in Q_M^n}\subset {\mathbb C}^{2^n\times 2^n}$ such that the trigonometric polynomial $U:\mathbb{R}^n\to\mathbb{C}^{2^n\times 2^n}$ given by (\ref{trig poly HD})
satisfies the following three conditions:
\begin{enumerate}
	\item[(i)] $U(\xi )$ is unitary for all $\xi\in\mathbb{R}^n$.
	\item[(ii)] $U(\xi +v_{2^j}/2)=\sigma_{2^j} U(\xi )$ for all $\xi\in\mathbb{R}$ and $0\leq j\leq n-1$ where $\sigma_{2^j}$ is as in (\ref{sigma def nD}).	
	\item[(iii)] $U(0)\in1\otimes {\mathcal U}(2^n-1)$.
	\end{enumerate}
To allow for regularity of the associated scaling function/wavelet pairs  we also impose 
\begin{enumerate}
	\item[(iv)] $\partial^\alpha U(\xi )\bigg|_{\xi =0}\in{\mathbb C}\otimes{\mathbb C}^{(2^n-1)\times (2^n-1)}$ for $1\leq |\alpha |\leq d$.
\end{enumerate}
\end{problem}
Conditions (i)--(iv) do not guarantee the orthogonality of the integer shifts of the scaling function. Bownik's sufficient condition for orthogonality may be written as follows:
\begin{enumerate}
	\item[(v)] $U(\xi )_{11}\neq 0$ for $\|\xi \|_\infty =\max_{1\leq j\leq n}|\xi _j|\leq 1/4$
\end{enumerate}
where $U(\xi )_{11}$ is the top left-hand entry of $U(\xi )$. Our algorithms are designed to find examples of sequences $\{A_k\}_{k\in Q_M^n}$ for which the function $U$ defined by (\ref{trig poly HD}) satisfies conditions (i)--(iv). After finding such an example, we discard it if (v) is not satisfied.

\subsection{Sampling and the discrete Fourier transform}\label{sec: sampling}
The assumption that the function $U=U(\xi )$ is a trigonometric polynomial allows for discretisation through sampling. We use this observation to recast conditions (i)-(iv) of Problem \ref{prob:p1} into constraints on a finite number of coefficient matrices $\{A_k\}_{k\in Q_M^n}$.

If $B,C\in{\mathbb C}^{N\times N}$, we define the inner product $\langle B,C\rangle$ by
$\langle B,C\rangle =\text{tr}(BC^*)=\sum_{i,j=1}^Nb_{ij}\overline{c_{ij}}$.
The norm arising from this inner product is the Frobenius norm $\|\cdot\|_2$. Let
$L^2([0,1]^n,{\mathbb C}^{N\times N})$ be the collection of measurable functions $F:[0,1]^n\to{\mathbb C}^{N\times N}$ for which $\int_{[0,1]^n}\|F(\xi )\|_2^2\, d\xi <\infty$.
Given $F,G\in L^2([0,1]^n,{\mathbb C}^{N\times N})$, we declare the inner product $\langle F,G\rangle$ to be
$$\langle F,G\rangle=\int_{[0,1]^n}\langle F(\xi ),G(\xi )\rangle\, d\xi .$$
The sequence space $\ell^2({\mathbb Z}^n,{\mathbb C}^{N\times N})$ is the collection of functions 
${\mathbf C}:{\mathbb Z}^n\to{\mathbb C}^{N\times N}$ for which $\sum_{k\in{\mathbb Z}^n}\|C_k\|_2^2<\infty$.
The inner product of ${\mathbf B}$ and ${\mathbf C}\in\ell^2({\mathbb Z}^n,{\mathbb C}^{N\times N})$ is given by  $\langle {\mathbf B},{\mathbf C}\rangle =\sum_{k\in{\mathbb Z}^n}\langle B_k,C_k\rangle$. The Fourier transform ${\mathcal F}:\ell^2({\mathbb Z}^n,{\mathbb C}^{N\times N})\to L^2([0,1]^n,{\mathbb C}^{N\times N})$  given by 
${\mathcal F}({\mathbf A})(\xi )=\sum_{k\in{\mathbb Z}^n} A_ke^{-2\pi i\langle k,\xi\rangle}$
is a unitary mapping with inverse ${\mathcal F}^{-1}$ given by 
$({\mathcal F}^{-1}G)_k=\int_{[0,1]^n}G(\xi )e^{2\pi i\langle k,\xi\rangle}\, d\xi$ whenever the integral converges.
 The space ${\mathcal T}_{M,n}^N$ of $N\times N$ matrix-valued trigonometric polynomials of degree less than $M-1$ is
$${\mathcal T}_{M,n}^N=\bigg\{P:{\mathbb R}^n\to{\mathbb C}^{N\times N};\ P(\xi )=\sum_{k\in Q_M^n}A_ke^{-2\pi i\langle k,\xi\rangle}\text{ for some }\{A_k\}_{k\in Q_M^n}\subset{\mathbb C}^{N\times N}\bigg\}$$
and the finite sequence space ${\mathcal X}_{M,n}^N$ is given by
$${\mathcal X}_{M,n}^N=\{{\mathbf C}\in\ell^2({\mathbb Z}^n,{\mathbb C}^{N\times N});\ C_k=0\text{ if }k\notin Q_M^n\}.$$
We note that ${\mathcal T}_{M,n}^N$ is a closed subspace of $L^2([0,1]^n,{\mathbb C}^{N\times N})$ and ${\mathcal X}_{M,n}^N$ is a closed subspace of $\ell^2({\mathbb Z}^n,{\mathbb C}^{N\times N})$. The Fourier transform may be restricted to ${\mathcal X}_{M,n}^N$, and in doing so it becomes a unitary mapping of ${\mathcal X}_{M,n}^N$ onto ${\mathcal T}_{M,n}^N$ which we continue to denote ${\mathcal F}$.

The orthogonal projection $P_{M,n}^N$ from $L^2([0,1]^n,{\mathbb C}^{N\times N})$ onto ${\mathcal T}_{M,n}^N$ is given by
$$P_{M,n}^NF(\xi )=\int_{[0,1]^n}K_M^N(\xi -\eta )F(\xi )\, d\eta$$
where 
$$K_M^N(\xi )=\begin{cases}e^{-\pi i(M-1)(\xi_1+\cdots +\xi_n)}\prod_{j=1}^n\frac{\sin (\pi M\xi_j)}{\sin (\pi\xi_j)}&\text{ if $\xi\neq 0$}\\
M^n&\text{ if $\xi =0$}\end{cases}$$
and the orthogonal projection $R_{M,n}^N$ from $\ell^2({\mathbb Z}^n,{\mathbb C}^{N\times N})$ onto ${\mathcal X}_{M,n}^N$ is given by
$$(R_{M,n}^N({\mathbf C}))_k=\begin{cases}C_k&\text{ if $k\in Q_M^n$}\\
0&\text{ else.}
\end{cases}$$

The sampling operator $D_{M,n}^N:{\mathcal T}_{M,n}^N\to({\mathbb C}^{N\times N})^{Q_M^n}$ is given by $(D_{M,n}^NP)_j=P(j/M)$ $(j\in Q_M^n)$ and there is an obvious isomorphism $\tau$ between $({\mathbb C}^{N\times N})^{Q_M^n}$ and ${\mathcal X}_{M,n}^N$, namely
$$(\tau{\mathbf A})_j=\begin{cases}
A_j&\text{ if $j\in Q_M^n$}\\
0&\text{ else.}
\end{cases}$$

The discrete Fourier transform ${\mathcal F}_M:({\mathbb C}^{N\times N})^{Q_M^n}\to({\mathbb C}^{N\times N})^{Q_M^n}$ is given by
$$({\mathcal F}_M{\mathbf B})_j=\sum_{k\in Q_M^n}B_ke^{-2\pi i\langle k,j\rangle /M},$$
with inverse   given by 
$({\mathcal F}_M^{-1}{\mathbf A})_k=M^{-n}\sum_{j\in M_n}A_je^{2\pi i\langle j,k\rangle /M}$.
Given $U\in{\mathcal T}_{M,n}^N$, we form the {\it matrix ensemble} ${\mathbf U}\in ({\mathbb C}^{N\times N})^{Q_M^n}$ by uniform sampling: the $j$-th entry of ${\mathbf U}$ is  $U_j=U(j/M)$ $(j\in Q_M^n)$, i.e., ${\mathbf U}=D_{M,n}^NU$. Furthermore, if $U(\xi )=\sum_{k\in Q_M^n}A_ke^{-2\pi i\langle k,\xi\rangle}$, then
\begin{equation}
U_j=U(j/M)=\sum_{k\in Q_M^n}A_ke^{-2\pi i\langle j,k\rangle /M};\quad A_k=\frac{1}{M^n}\sum_{j\in Q_M^n}U_je^{2\pi i\langle j,k\rangle /M},\label{FT pair}
\end{equation}
i.e., the ensembles ${\mathbf U}$ and ${\mathbf A}$ form a (finite) Fourier transform pair. For this reason, properties of $U=U(\xi )$ may be encoded into its samples $U_j=U(j/M)$ $(j\in Q_M^n)$ by way of its coefficients $A_k$ $(k\in Q_M^n)$. Written in the ``ensemble'' notation, we denote the finite Fourier transform operations of equation (\ref{FT pair}) as follows:
\begin{equation*}
{\mathbf U}={\mathcal F}_M{\mathbf A};\quad {\mathbf A}=({\mathcal F}_M)^{-1}{\mathbf U}=({\mathcal F}_M)^{-1}D_M^NU;\quad {U}={\mathcal F}{\mathbf A}.
\end{equation*}
These relationships are summarised in the following commuting diagram.
\[
  \begin{tikzcd}
    {\mathbf A}\in{\mathcal X}_{M,n}^N \arrow{r}{{\mathcal F}} \arrow[swap]{dr}{{\mathcal F}_M} & U\in{\mathcal T}_{M,n}^N \arrow{d}{D_{M,n}^N} \\
     & {\mathbf U}\in({\mathbb C}^{N\times N})^{Q_M^n}
  \end{tikzcd}
\]

\subsection{Discretisation}
Sampling and the discrete Fourier transform provide a means through which Problem  \ref{prob:p1} may be discretised in the sense that the construction of a matrix-valued function $U(\xi )$ $(\xi\in{\mathbb R}^n)$ satisfying the conditions of Problem \ref{prob:p1} may be replaced by the construction of finitely many matrices $U_j$ satisfying a compatible collection of conditions.

\subsubsection{Consistency}
The consistency condition (ii) of Problem \ref{prob:p1} can be written in terms of the coefficient matrices $\{A_k\}_{k\in Q_M^n}$ or the sampled matrices $\{U_j=U(j/M)\}_{j\in Q_M^n}$.

\begin{proposition}\label{prop: U-A consistency} Let $P$ be the trigonometric polynomial $P(\xi )=\sum_{k\in Q_M^n}A_ke^{-2\pi i\langle k,\xi\rangle}$ $(\{A_k\}_{k\in Q_M^n}\subset{\mathbb C}^{2^n\times 2^n}$) and $P_j=P(j/M)$ $(j\in Q_M^n)$. Then the following are equivalent:
\begin{enumerate}
\item[(i)] $P(\xi +v_{2^\ell}/2)=\sigma_{2^\ell} P(\xi )$ for all $\xi\in{\mathbb R}^n$ and all $\ell\in\{0,1,\dots ,n-1\}$
\item[(ii)] $\sigma_{2^\ell}A_k=(-1)^{k_\ell}A_k$ for all $k=(k_1,k_2,\dots ,k_n)\in Q_M^n$ and all $\ell\in\{0,1,\dots ,n-1\}$
\item[(iii)] $P_{j+Mv_{2^\ell}/2}=\sigma_{2^\ell} P_j$ for all $j\in Q_M^n$ and all $\ell\in\{0,1,\dots ,n-1\}$.
\end{enumerate}
\end{proposition}

\begin{proof} Suppose $P$ satisfies the consistency condition (i). Then 
\begin{align*}
\sum_{k\in Q_M^n}A_k(-1)^{k_\ell}e^{-2\pi i\langle k,\xi\rangle}&=\sum_{k\in Q_M^n}A_ke^{-2\pi i\langle k,\xi+v_{2^\ell}/2\rangle}\\
&=P(\xi +v_{2^\ell}/2)=\sigma_{2^\ell} P(\xi )=\sum_{k\in Q_M^n}\sigma_{2^\ell} A_ke^{-2\pi i\langle k,\xi\rangle}.
\end{align*}
Comparing coefficients in the sums on both sides of this equality gives $\sigma_{2^\ell} A_k=(-1)^{k_\ell}A_k$, hence (i) $\Rightarrow$ (ii). A similar calculation gives the converse. Now suppose $\{A_k\}_{k\in Q_M^n}$ satisfies (ii). Then 
\begin{align*}
P_{j+Mv_{2^\ell}/2}&=\sum_{k\in Q_M^n}A_ke^{-2\pi i\langle k,j+Mv_{2^\ell /2}\rangle /M}\\
&=\sum_{k\in Q_M^n}A_k(-1)^{k_\ell}e^{-2\pi i\langle j,k\rangle /M}=\sigma_{2^\ell}\sum_{k\in Q_M^n}A_ke^{-2\pi i\langle j,k\rangle /M}=\sigma_{2^\ell} P_j
\end{align*}
so that (ii)$\Rightarrow$(iii). The converse is proved similarly.
\end{proof}

\subsubsection{Orthogonality/Unitarity} 
The discretisation of the problem of constructing wavelet matrices $U(\xi )$ in $n$ dimensions relies on the fact that the sampling operator $D_{M,n}^N:{\mathcal T}_{M,n}^N\to{\mathcal X}_{M,n}^N$ is a multiple of a unitary operator. As we saw at the start of this section, the orthogonality of the collections $\{\tau_k\varphi\}_{k\in{\mathbb Z}^n}$ and $\{\tau_k\psi_\varepsilon\}_{k\in{\mathbb Z}^n}$ $(1\leq\varepsilon\leq 2^n-1)$ and the orthogonality of the spaces they span requires that $U(\xi )$ as given in (\ref{paraunitarity HD})  be unitary for all $\xi$. 
It is not sufficient to impose unitarity of the samples $U_j=U(j/M)$. To see this, consider the one-dimensional example 
$$U(\xi )=\frac{1}{4}[2(I_2+\sigma)+(1+i)(I_2-\sigma )e^{-2\pi i\xi}+(1-i)(I_2-\sigma )e^{-6\pi i\xi}]\quad (\xi\in {\mathbb R}).$$
Here $I_2$ is the $2\times 2$ identity matrix and $\sigma =\left(\begin{matrix}0&1\\1&0\end{matrix}\right)$. Since $\sigma^2=I_2$, $U$ satisfies the consistency condition $U(\xi +1/2)=\sigma U(\xi )$. Furthermore, $U(0)=U(1/4)=I_2$ while $U(1/2)=U(3/4)=\sigma$, all of which are unitary, yet 
$U(1/8)=\dfrac{1}{2}(I_2+\sigma )$ which is not unitary since $U(1/8)^*U(1/8)=\dfrac{1}{2}(I_2+\sigma )$.

\begin{proposition}\label{prop: discrete unitary} Let $\{A_k\}_{k\in Q_M^n}\subset{\mathbb C}^{2^n\times 2^n}$. The trigonometric polynomial  $U(\xi )=\sum_{k\in Q_M^n}A_ke^{-2\pi i\langle k,\xi\rangle}$, $(\xi\in{\mathbb R}^n)$ is unitary for all $\xi\in{\mathbb R}^n$ if and only if $U(j/(2M))$ is unitary for all $j\in Q_{2M}^n$.
\end{proposition}

\begin{proof} Let $J_M^n=\{1-M,\dots ,0,\dots ,M-1\}^n$. If $U$ is unitary for all $\xi$, then it is clearly unitary at all $\xi\in Q_{2M}^n/(2M)$. Note that for all $\xi\in{\mathbb R}^n$,
\begin{equation}
U(\xi )^*U(\xi )=\sum_{m\in J_M^n}B_me^{-2\pi i\langle m,\xi\rangle}\label{continuous unitarity}
\end{equation}
with $B_m=\sum_{k\in Q_M^n}A_k^*A_{m+k}$ $(m\in J_M^n)$.  Suppose now that $U$  is unitary at all points of $Q_{2M}^n/(2M)$, i.e., $U(j/(2M))$ is unitary for all $j\in Q_{2M}^n$. Then for all $j\in Q_{2M}^n$ we have
\begin{equation}
I_n=U(j/(2M))^*U(j/(2M))=\sum_{m\in J_M^n}B_me^{-2\pi i\langle m,j\rangle /(2M)}.\label{discrete unitarity}
\end{equation}
By the orthonormality and completeness of the Fourier basis $\{e_m\}_{m\in Q_{2M}^n}$ (where $e_m(j)=e^{-2\pi i\langle m,j\rangle/(2M)}$) in $\ell^2(J_M^n,{\mathbb C})$, we conclude from (\ref{discrete unitarity}) that $B_m=\delta_{m,0}I_n$ and from (\ref{continuous unitarity}) that $U(\xi )^*U(\xi )=I_n$ for all $\xi\in{\mathbb R}^n$.
\end{proof}

Note that Proposition \ref{prop: discrete unitary} involves the sampled ensemble $\{U(j/(2M))\}_{j\in Q_{2M}^n}$ rather than $\{U(j/M)\}_{j\in Q_M^n}$. In Section \ref{sec: wavelet feasibility}, we'll see that since
$$Q_{2M}^n/(2M)=\cup_{k=0}^{2^n-1}(Q_M^n+v_{k}/2)/M,$$
the  ensemble $\{U(j/(2M))\}_{j\in Q_{2M}^n}$ can be computed from ${\mathbf U}=\{U(j/M)\}_{j\in Q_M^n}$, so that unitarity of $U(\xi )$ at all $\xi$ can be achieved by the imposition of appropriate conditions on $\{U(j/M)\}_{j\in Q_M^n}$.

\subsubsection{Regularity} The regularity condition (iv) of Problem \ref{prob:p1} can be written in terms of the coefficient matrices $\{A_k\}_{k\in Q_M^n}$ or the sampled matrices $\{U_j=U(j/M)\}_{j\in Q_M^n}$.

\begin{proposition} \label{prop: discrete regularity}Let $P$ be the trigonometric polynomial $P(\xi )=\sum_{k\in Q_M^n}A_ke^{-2\pi i\langle k,\xi\rangle}$ $(\{A_k\}_{k\in Q_M^n}\subset{\mathbb C}^{N\times N}$), $P_j=P(j/M)$ $(j\in Q_M^n)$ and $\alpha =(\alpha_1,\dots ,\alpha_N)\in{\mathbb Z}_+^N$. Then the following are equivalent:
\begin{enumerate}
\item[(i)] $\partial^\alpha P(\xi )\bigg|_{\xi =0}\in{\mathbb C}\otimes{\mathbb C}^{(N-1)\times (N-1)}$
\item[(ii)] $\sum_{k\in Q_M^n}k^\alpha A_k\in{\mathbb C}\otimes{\mathbb C}^{(N-1)\times (N-1)}$
\item[(iii)] $\sum_{j\in Q_M^n}c_{\alpha j}P_j\in{\mathbb C}\otimes{\mathbb C}^{(N-1)\times (N-1)}$ where $c_{\alpha j}=\sum_{k\in Q_M^n}k^\alpha e^{2\pi i\langle j,k\rangle /M}$.
\end{enumerate}
\end{proposition}

\begin{proof} We have $A_k=\dfrac{1}{M^n}\sum_{j\in Q_M^n}P_je^{2\pi i\langle j,k\rangle /M}$. Furthermore,
$$\partial^\alpha P(\xi )=(-2\pi i)^{|\alpha |}\sum_{k\in Q_M^n}k^\alpha A_ke^{-2\pi i\langle k,\xi\rangle}$$
so that 
$$\sum_{k\in Q_M^n}k^\alpha A_k=\frac{1}{M^n}\sum_{j\in Q_M^n}P_j\sum_{k\in Q_M^n}k^\alpha e^{2\pi i\langle j,k\rangle /M}=\frac{1}{M^n}\sum_{j\in Q_M^n}c_{\alpha j}P_j$$
with $c_{\alpha j}$ as in the statement of the proposition.
\end{proof}

The wavelet construction problem has now been recast as follows:

\begin{problem}\label{matrix construction} Given an even integer $M\geq 4$, we seek a matrix ensemble ${\mathbf U}\in({\mathbb C}^{2^n\times 2^n})^{Q_M^n}$ such that 
\begin{enumerate}
\item[(i)] the matrix ensembles ${\mathbf U}_\ell =\bigg\{U\bigg(\dfrac{j}{M}+\dfrac{v_{\ell}}{2M}\bigg)\bigg\}_{j\in Q_M^n}$ $(0\leq\ell\leq 2^n-1)$ are unitary;
\item[(ii)] $U_{j+Mv_{2^\ell}/2}=\sigma_{2^\ell} U_j$ for all $j\in Q_{M/2}^n$;
\item[(iii)] $U_0\in 1\otimes {\mathcal U}(2^n-1)$.
\end{enumerate}
To allow for regularity of the associated scaling function and wavelets, we also impose
\begin{enumerate}
\item[(iv)] $\sum_{j\in Q_M^n}c_{\alpha j}U_j\in{\mathbb C}\otimes{\mathbb C}^{(2^n-1)\times (2^n-1)}$ for $1\leq|\alpha |\leq d$ where $c_{\alpha j}=\sum_{k\in Q_M^n}k^\alpha e^{2\pi i\langle j,k\rangle /M}$.
\end{enumerate}
\end{problem}

\subsection{Wavelet feasibility problem}\label{sec: wavelet feasibility}
Let $M$ be even, $d$ a non-negative integer, and  $\sigma_j$ be  the permutation matrix of (\ref{sigma def nD}).  We define
\begin{equation}
({\mathbb C}^{2^n\times 2^n})^{Q_M^n}_\sigma=\{{\mathbf U}\in ({\mathbb C}^{2^n\times 2^n})^{Q_M^n}:\, U_{j+Mv_{2^\ell} /2}=\sigma_{2^\ell}U_j,\ (j\in Q(M/2)^n,\ 0\leq\ell\leq n-1)\}\label{Csigma def nD}
\end{equation}
to be the collection of $\sigma$-consistent ensembles. 
It is straightforward to verify that $\CsigmaMn$ is a vector space over ${\mathbb C}$ under the usual componentwise operations, and $\CsigmaMn_\sigma$ is a vector subspace. Moreover, by Proposition \ref{prop: U-A consistency} we have
\begin{equation}
{\mathcal F}_M^{-1}\CsigmaMn_\sigma =\{{\mathbf A}\in\CsigmaMn;\ (-1)^{k_\ell}A_k=\sigma_{2^\ell} A_k,\ k\in Q_M^n\}.\label{Fourier CsigmaMn}
\end{equation}
We note that in the Fourier-side description (\ref{Fourier CsigmaMn}) of $\CsigmaMn_\sigma$, the condition applies individually to each of the matrices $A_k$ of the ensemble ${\mathbf A}$ rather than on certain pairs of matrices as in (\ref{Csigma def nD}).

\begin{lemma}\label{lem: off-centre ensembles}
Let $U$ be the trigonometric polynomial $U(\xi )=\sum_{k\in Q_M^n}A_ke^{-2\pi i\langle k,\xi\rangle}$ $(\xi\in{\mathbb R}^n)$ with $\{A_k\}_{k\in Q_M^n}\subset{\mathbb C}^{2^n\times 2^n}$ and ${\mathbf U}$ be the matrix ensemble with $j$-th term $U_j=U(j/M)$ $(j\in Q_M^n)$. Then 
$$U((j+v_\ell /2)/M)=({\mathcal F}_M\chi_\ell{\mathcal F}_M^{-1}{\mathbf U})_j$$
where $(\chi_\ell{\mathbf A})_k=e^{-\pi i\langle k,v_\ell\rangle /M}A_k$ $(k\in Q_M^n)$.
\end{lemma}

\begin{proof} Observe that 
\begin{align*}
U((j+v_\ell /2)/M)&=\sum_{k\in Q_M^n}A_ke^{-2\pi i\langle k,j+v_\ell /2\rangle /M}\\
&=\sum_{k\in Q_M^n}e^{-\pi i\langle k,v_\ell\rangle /M}({\mathcal F}_M^{-1}{\mathbf U})_ke^{-2\pi i\langle j,k\rangle /M}\\
&=\sum_{k\in Q_M^n}(\chi_\ell{\mathcal F}_M^{-1}{\mathbf U})_ke^{-2\pi i\langle k,j\rangle /M}=({\mathcal F}_M\chi_\ell{\mathcal F}_M^{-1}{\mathbf U})_j.
\end{align*}
\end{proof}

\subsubsection{Unitarity}
Lemma \ref{lem: off-centre ensembles} allows us to rephrase Proposition \ref{prop: discrete unitary} as follows:

\begin{proposition}\label{prop: unitary ensembles}
Let $U(\xi )=\sum_{k\in Q_M^n}A_ke^{-2\pi i\langle k,\xi\rangle}$ $(A_k\in{\mathbb C}^{2^n\times 2^n})$ be a trigonometric polynomial, ${\mathbf U}\in\CsigmaMn$ be the matrix ensemble with $j$-th entry $U_j=U(j/M)$ and 
$${\mathbf U}^{(\ell )}={\mathcal F}_M\chi_\ell{\mathcal F}_M^{-1}{\mathbf U}\in\CsigmaMn\quad (0\leq\ell\leq 2^n-1).$$
Then $U(\xi )$ is unitary for all $\xi\in{\mathbb R}^n$ if and only if the matrix ensembles $\{{\mathbf U}^{(\ell )}\}_{\ell =0}^{2^n-1}$ are all unitary.
\end{proposition}
It's important to note that if ${\mathbf U}\in\CsigmaMn_\sigma$ and $\{U_j\}_{j\in Q_{M/2}^n}$ are unitary, then all entries of ${\mathbf U}$ are unitary since for $j\in Q_{M/2}^n$, we have
\begin{equation*}
(U_{j+Mv_{2^\ell} /2})^*U_{j+Mv_{2^\ell} /2}=(\sigma_{2^\ell} U_j)^*(\sigma_{2^\ell}U_j)=U_j^*\sigma_{2^\ell}^*\sigma_{2^\ell} U_j=U_j^*U_j=I_{2^n}.
\end{equation*}
Therefore, when imposing unitarity on entries of an ensemble ${\mathbf U}\in\CsigmaMn_\sigma$, it is enough to impose unitarity on the sub-ensemble $\{U_j\}_{j\in Q_{M/2}^n}$.

\subsubsection{Consistency}

Given an ensemble ${\mathbf V}\in ({\mathbb C}^{2^n\times 2^n})^{Q_M^n}$, we extend it to a periodic mapping ${\mathbf V}:{\mathbb Z}^n\to ({\mathbb C}^{2^n\times 2^n})^{Q_M^n}$ by declaring
$V_{j+Mp}=V_j$ $(j\in Q_M^n,\ p\in{\mathbb Z}^n)$.
We consider translation operators ${\tau_k}$ ($k\in{\mathbb Z}^n$) acting on $({\mathbb C}^{2^n\times 2^n})^{Q_M^n}$ by
\begin{equation}
(\tau_k{\mathbf V})_j=V_{j+k}\quad (j,k\in{\mathbb Z}^n).\label{tau def}
\end{equation}
If $r$ is an integer with $0\leq r\leq n-1$, we define an operator $T_{2^r}$ on $({\mathbb C}^{2^n\times 2^n})^{Q_M^n}$ by
\begin{equation}
(T_{2^r}{\mathbf V})_j=\frac{2}{M}\sum_{m_r=0}^{M-1}\frac{V_{(j_1,\dots ,j_{r-1},m_r ,j_{r +1},\dots ,j_n)}}{1-e^{2\pi i(j_r-m_r-1/2)/M}}.\label{T_{2^r} def}
\end{equation}
If $\ell\in Y_n$ has binary expansion $\ell =\sum_{r=0}^{n-1}a_r2^r$ $(a_r\in\{0,1\})$ then we define
\begin{equation}
T_\ell=\prod_{r=0}^{n-1}(T_{2^r})^{a_r}.\label{T_l def}
\end{equation}
Since the operators $\{T_{2^r}\}_{r=0}^{n-1}$ commute, the product in (\ref{T_l def}) is well-defined.

\begin{lemma} \label{lem: tau T comm}As operators acting on periodisations of ensembles in $({\mathbb C}^{2^n\times 2^n})^{Q_M^n}$,  $\tau_m$ and $T_\ell$ $(m\in{\mathbb Z}^n,\ \ell\in Y_n)$ commute, i.e., $\tau_mT_\ell =T_\ell\tau_m$.
\end{lemma}

\begin{proof}
If $0\leq r\neq p\leq n-1$ and $s\in{\mathbb Z}$, then  
$$(\tau_{sv_{2^p}}T_{2^r}{\mathbf V})_j=\frac{2}{M}\sum_{m_r=0}^{M-1}\frac{V_{(j_1,\dots ,j_{r-1},m_r ,j_{r +1},\dots , j_{p}+s,\dots ,j_n)}}{1-e^{-2\pi i(j_r -m_r +1/2)/M}}=(T_{2^r}\tau_{sv_{2^p}}{\mathbf V})_j$$
while if $0\leq r=p\leq n-1$ and $s\in{\mathbb Z}$,
\begin{align*}
(\tau_{sv_{2^r}}T_{2^r}{\mathbf V})_j&=\frac{2}{M}\sum_{m_r=0}^{M-1}\frac{V_{(j_1,\dots ,j_{r-1},m_r ,j_{r +1},\dots ,j_n)}}{1-e^{-2\pi i(j_r+s-m_r+1/2)/M}}\\
&=\frac{2}{M}\sum_{m_r=0}^{M-1}\frac{V_{(j_1,\dots ,j_{r-1},m_r+s ,j_{r +1},\dots ,j_n)}}{1-e^{-2\pi i(j_r-m_r+1/2)/M}}=(T_{2^r}\tau_{sv_{2^r}}{\mathbf V})_j.
\end{align*}
We conclude that $\tau_{sv_{2^p}}T_{2^r}=T_{2^r}\tau_{sv_{2^p}}$ for all $0\leq p,r\leq n-1$. Hence, if $\ell\in Y_n$ and $T_\ell$ is defined as in (\ref{T_l def}), we have
$$\tau_{sv_{2^p}}T_\ell=\tau_{sv_{2^p}}\prod_{r=0}^{n-1}(T_{2^r})^{a_r}=\prod_{r=0}^{n-1}(T_{2^r})^{a_r}\tau_{sv_{2^p}}=T_\ell\tau_{sv_{2^p}}.$$
Finally, if $m=\sum_{p=0}^{n-1}s_pv_{2^p}\in{\mathbb Z}^n$ then
$$\tau_mT_\ell=\prod_{p=0}^{n-1}\tau_{s_pv_{2^p}}T_\ell=T_\ell\prod_{p=0}^{n-1}\tau_{s_pv_{2^p}}=T_\ell\tau_m.$$
\end{proof}

\begin{proposition} Suppose ${\mathbf U}\in\CsigmaMn_\sigma$ and ${\mathbf U}^{(\ell )}={\mathcal F}_M\chi_\ell{\mathcal F}_M^{-1}{\mathbf U}$ for some $0\leq\ell\leq 2^n-1$. Then ${\mathbf U}^{(\ell )}$ satisfies the consistency condition, i.e., ${\mathbf U}^{(\ell )}\in\CsigmaMn_\sigma$.
\end{proposition}

\begin{proof} If $U_j^{(\ell )}$ is the $j$-th component of ${\mathbf U}^{(\ell )}$, then with $j=(j_1,\dots ,j_n)\in Q_M^n$, $U_j=U_{(j_1,\dots ,j_n)}$ and $0\leq r\leq n-1$,
\begin{align}
U^{(2^r)}_j&=({\mathcal F}_M^{-1}\chi_{2^r}{\mathcal F}_M{\mathbf U})_j\notag\\
&=\frac{1}{M^n}\sum_{k\in Q_M^n}\sum_{m\in Q_M^n}U_me^{-2\pi i\langle m,k\rangle /M}e^{2\pi i\langle k,j\rangle /M}e^{-\pi ik_r /M}\notag\\
&=\frac{1}{M^n}\sum_{m\in Q_M^n}U_m\sum_{k_r=0}^{M-1}e^{2\pi ik_r(j_r-m_r-1/2)/M}\prod_{\substack{p=1\\p\neq r}}^n\sum_{k_p=0}^{M-1}e^{2\pi ik_p(j_p-m_p)/M}\notag\\
&=\frac{1}{M}\sum_{m\in Q_M^n}U_m\delta_{j_1-m_1}\cdots\delta_{j_{r-1}-m_{r-1}}\delta_{j_{r+1}-m_{r+1}}\dots\delta_{j_n-m_n}\frac{-2}{e^{2\pi i(j_r-m_r-1/2)/M}-1}\notag\\
&=\frac{2}{M}\sum_{m_r =0}^{M-1}\frac{U_{(j_1,\dots ,j_{r -1},m_r ,j_{r +1},\dots ,j_n)}}{1-e^{2\pi i(j_r -m_r-1/2)/M}}=(T_{2^r}{\mathbf U})_j,\label{Uli formula}
\end{align}
i.e., ${\mathbf U}^{(2^r)}=T_{2^r}{\mathbf U}$ with $T_{2^r}$ the convolution operator defined in (\ref{T_{2^r} def}). If $\ell =\sum_{r=0}^{n-1}a_r2^r$ $(a_r\in\{0,1\})$, then $T_\ell$ is defined as in (\ref{T_l def}) and by (\ref{Uli formula}) we have
\begin{align*}
{\mathbf U}^{(\ell)}={\mathcal F}_M\chi_\ell{\mathcal F}_M^{-1}{\mathbf U}
&={\mathcal F}_M\prod_{r=0}^{n-1}(\chi_{2^r})^{a_r}{\mathcal F}_M^{-1}{\mathbf U}\\
&=\prod_{r=0}^{n-1}({\mathcal F}_M\chi_{2^r}{\mathcal F}_M^{-1})^{a_r}{\mathbf U}=\prod_{r=0}^{n-1}(T_{2^r})^{a_r}{\mathbf U}=T_\ell{\mathbf U}.\end{align*}
Hence, if ${\mathbf U}\in({\mathbb C}^{2^n\times 2^n})^{Q_M^n}_\sigma$ is a consistent ensemble, $\ell\in Y_n$ and $0\leq p\leq n-1$, Lemma \ref{lem: tau T comm} gives
$$U^{(\ell )}_{j+Mv_{2^p}/2}=(\tau_{Mv_{2^p}/2}T^\ell{\mathbf U})_j=(T^\ell \tau_{Mv_{2^p}/2}{\mathbf U})_j=(T^\ell\sigma_{2^p}{\mathbf U})_j=\sigma_{2^p}(T^\ell{\mathbf U})_j=\sigma_{2^p}U^{(\ell )}_j,$$
i.e., ${\mathbf U}^{(\ell )}$ is a consistent ensemble.
\end{proof}

We close this section with the discretised version of the wavelet construction problem:

\begin{problem}\label{prob:discrete wavelet const prob}
Given an even integer $M\geq 4$, we seek a matrix ensemble ${\mathbf U}\in({\mathbb C}^{2^n\times 2^n})^{Q_M^n}_\sigma$ such that 
\begin{enumerate}
\item[(i)] ${\mathbf U}^{(\ell )}=\{{\mathcal F}_M\chi_\ell{\mathcal F}_M^{-1}{\mathbf U}\}_{\ell =1}^{2^n-1}$ are unitary ensembles;
\item[(ii)] $U_0\in 1\otimes {\mathcal U}(2^n-1)$.
\end{enumerate}
To allow for regularity of the associated scaling function and wavelets, we also impose
\begin{enumerate}
\item[(iii)] $\sum_{j\in Q_M^n}c_{\alpha j}U_j\in{\mathbb C}\otimes{\mathbb C}^{(2^n-1)\times (2^n-1)}$ for $1\leq|\alpha |\leq d$ where $c_{\alpha j}=\sum_{k\in Q_M^n}k^\alpha e^{2\pi i\langle j,k\rangle /M}$.
\end{enumerate}
\end{problem}

\section{Projection algorithms}  \label{sec: opt prelim}
In this section we give the background required to solve Problem \ref{prob:discrete wavelet const prob} with techniques borrowed from optimisation.
\subsection{Projection operators}
Let ${\mathcal H}$ be a finite-dimensional Hilbert space. Given a set $S\subseteq{\mathcal H}$, its \emph{(metric) projector} is the set-valued operator given by 
  $$ P_S(x) := \left\{s\in S:\|s-x\|\leq d(x,S)\right\}\quad (x\in{\mathcal H}) $$
where $d(x,S)=\inf_{s\in S}\|x-s\|$. It is straightforward to check that $P_S(x)\neq\emptyset$ for all $x\in{\mathcal H}$ so long as $S$ is nonempty and closed. In a common abuse of notation, we write $P_S(x)=p$ to mean $P_S(x)=\{p\}$.

\begin{proposition}[Properties of projectors]\label{prop:proj properties} Let ${\mathcal H}$ be a finite dimensional Hilbert space.
\begin{enumerate}[(a)]
  \item\label{it:proj product} Let $C_1,C_2,\dots,C_{m}\subseteq{\mathcal H}$ be nonempty closed sets and define $C:=C_1\times\dots\times C_{m}\subseteq{\mathcal H}^m$. Then $$P_C=P_{C_1}\times\dots\times P_{C_{m}}. $$
  \item\label{it:proj isometry} Let $L:{\mathcal H}\to{\mathcal H}$ be an isometric isomorphism and $C\subseteq{\mathcal H}$ be a nonempty closed set. Then $$ P_{L(C)} = L\circ P_C\circ L^{-1}.$$
\end{enumerate}
\end{proposition}
\begin{proof}
\eqref{it:proj product}:~Follows easily from the definition.

\noindent \eqref{it:proj isometry}:~Let $x\in{\mathcal H}$. First note that since $L$ is an isometric isomorphism, we have $d(x,L(C)) = d(L^{-1}x,C)$.
On one hand, if $p\in P_{L(C)}(x)$, then $L^{-1}p\in C$ and 
  $$ d(L^{-1}x,C)=d(x,L(C)) = \|x-p\| = \|L^{-1}x-L^{-1}c\|. $$
This implies that $L^{-1}p\in P_C(L^{-1}x)$ or, equivalently, that $p\in (L\circ P_C\circ L^{-1})(x)$. On the other hand, if $p\in (L\circ P_C\circ L^{-1})(x)$, then there exists $c\in P_C(L^{-1}x)$ such that $p=Lc$ and
  $$ d(x,L(C))=d(L^{-1}x,C)=\|L^{-1}x-c\|=\|x-Lc\|=\|x-p\|,$$
which implies that $p\in P_{L(C)}(x)$. This completes the proof.
\end{proof}

In what follows, the unit sphere is denoted $\mathbb{S}:=\{x\in{\mathcal H}:\|x\|=1\}$. We recall that the singular value decomposition (SVD) of a matrix $A\in{\mathbb C}^{N\times N}$ is of the form $A=U\Sigma V^*$ where $U,V\in {\mathcal U}(N)$ and $\Sigma\in {\mathbb C}^{N\times N}$ is a diagonal matrix with the diagonal entries (the {\it singular values} of $A$) being the eigenvalues of $\sqrt{A^*A}$.

\begin{proposition}[Examples of projectors]\label{prop:proj examples} Let ${\mathcal H},{\mathcal H}'$ be finite dimensional Hilbert spaces.
\begin{enumerate}[(a)]
  \item\label{it:proj linear equalities} Let $L:{\mathcal H}\to{\mathcal H}'$ be linear and denote $C:=\{x\in{\mathcal H}:Lx=0\}$. If $LL^*$ is invertible, then $$ P_C(x)=x-L^*(LL^*)^{-1}(Lx)\quad \forall x\in{\mathcal H}. $$
  \item\label{it:proj sphere} Let $x\in{\mathcal H}$. Then 
    $P_{\mathbb{S}}(x) = \begin{cases} 
                            \frac{x}{\|x\|} & x\neq 0, \\
                             \mathbb{S}     & x = 0.\\
                         \end{cases}$
  \item\label{it:proj unitary} Let $X\in{\mathbb C}^{N\times N}$. Then 
     $P_{{\mathcal U}(N)}(X)=\{UV^*:X=U\Sigma V^*\text{ is an SVD}\}.$
	\end{enumerate}
\end{proposition}  

\begin{proof}
	\eqref{it:proj linear equalities}:~See \cite[Example~28.14(iii)]{BauCom}.
	
	\noindent \eqref{it:proj sphere}:~Follows easily from the definitions.
	
	\noindent \eqref{it:proj unitary}:~See \cite[Theorems~8.1~\&~8.6]{Higham}.
\end{proof}

We note that if $\sigma\in{\mathcal U}(N)$ and $X\in{\mathbb C}^{N\times N}$ then 
\begin{equation}
P_{{\mathcal U}(N)}(\sigma X)=\sigma P_{{\mathcal U}(N)}(X).\label{unitary proj mult comm}
\end{equation}

\subsection{Projection Algorithms and Feasibility Problems}\label{sec: projection algorithms}
Given finitely many closed sets $C_1,\dots,C_m\subseteq{\mathcal H}$ (a finite-dimensional Hilbert space) with nonempty intersection, the corresponding \emph{feasibility problem} is
  \begin{equation}
  \label{eq:feasibility problem}
    \text{find~}x\in \bigcap_{k=1}^mC_k.
  \end{equation}   
\emph{Projection algorithms} are a family of iterative algorithms which can be used to solve \eqref{eq:feasibility problem} by in each step utilising only projectors onto the individual sets (rather than the entire intersection at once). The two most important examples of projection algorithms are the \emph{method of cyclic projections} \cite{breg} and the \emph{Douglas--Rachford (DR) method} \cite{LionMercier,BauComLuk}, as well as their variants \cite{BorTam,AraCam}. 

In this work we employ the \emph{Douglas--Rachford method} which can be compactly described as the following fixed point iteration: Given $x_0\in{\mathcal H}$, choose any sequence $(x_k)$ satisfying
  \begin{equation}
  \label{eq:DR}
  x_{k+1} \in T(x_k) \text{ where }T:=\frac{I+R_CR_D}{2},
  \end{equation}
and $R_A:=2P_A-I$ denotes \emph{reflector} with respect to a set $A$. Here we note that the sequence $(x_k)$ is only required to satisfy the inclusion in (\ref{eq:DR}) since, in general, the operator $T:{\mathcal H}\to2^{{\mathcal H}}$ is a point-to-set mapping.

When applying a method based on \eqref{eq:DR}, the sequence of interest (\emph{i.e.,} the one that solves \eqref{eq:feasibility problem}) is not $(x_k)$ itself, but one of its projections onto the set $D$. For this reason, it is convenient to implement the Douglas--Rachford algorithm as outlined in Algorithm~\ref{alg:DR} and, in order to be concrete, we state a general convergence result for the convex setting in Theorem~\ref{th:DR convex}.

\begin{figure}[htb]
\begin{algorithm}[H]
\caption{Implementation of the Douglas--Rachford algorithm.\label{alg:DR}}	
\textbf{Input:~}$x_0\in{\mathcal H}$\;
Set $k:=0$ and choose any $p_0\in P_D(x_0)$\;
\While{stopping criteria not satisfied}{
Choose any point $x_{k+1}$ satisfying
\begin{align*}
  x_{k+1} &\in x_k+P_C(2p_k-x_k)-p_k; \\
\intertext{Choose any point $p_{k+1}$ satisfying}  
  p_{k+1} &\in P_D(x_{k+1});
\end{align*}
Set $k:=k+1$\;
}
\textbf{Return:~}$p_k$\;
\end{algorithm}
\end{figure}

Although Algorithm~\ref{alg:DR} applies to problem \eqref{eq:feasibility problem} with $n=2$, the general problem \eqref{eq:feasibility problem} can always be cast as a two set problem via the following product space formulation. Let $C$, $D$ be subsets of ${\mathcal H}^m$ given by 
  \begin{align*}
    C := C_1\times C_2\times\dots\times C_m, \qquad 
    D := \{(x,x,\dots,x)\in{\mathcal H}^m:x\in{\mathcal H}\}. 
  \end{align*}
Then the following equivalence holds:
  $$ x\in\bigcap_{k=1}^mC_k   \iff (x,x,\dots,x)\in C\cap D. $$
From here onwards, when speaking of applying the Douglas--Rachford algorithm to a feasibility problem, we will always mean its product space reformulation.
\begin{theorem}[Behaviour of the DR algorithm {\cite[Theorem~3.13]{BauComLuk}}]\label{th:DR convex}
	Suppose $C,D\subseteq{\mathcal H}$ are closed and convex with nonempty intersection. Let $x_0\in{\mathcal H}$ and set $x_{k+1}=T(x_{k})$ for all $k\in\mathbb{N}$. Then  the sequence $(x_k)$ converges to a point $x\in\Fix T:=\{x:Tx=x\}$ and, moreover, $P_D(x)\in C\cap D$.
\end{theorem}

In general, beyond the case of convex sets there is insufficient theory to justify application of projection methods. Indeed, most non-convex results in the literature rely on restrictive regularity notions from nonsmooth analysis and, even then, only yield local convergence guarantees \cite{HLN,Phan,DaoTam}. Nevertheless, projection methods have been empirically observed to still perform reasonably well in certain non-convex settings include matrix completion \cite{anziam}, graph colouring \cite{graphcolouring}, combinatorial optimization \cite{combdesign,comboptim}, road design \cite{roaddesign}, and constraint satisfaction \cite{GraEls}. This experience suggests use of the Douglas--Rachford method in the setting outlined in the following section.

\subsection{Hilbert spaces of matrix ensembles}
Although the matrices we work with have complex entries, for the purpose of algorithms is more convenient to work in a space over the real field. In this section, we provide the necessary background to justify this process. Before doing so, we first recall that the \emph{Frobenius inner-product} on ${\mathbb C}^{N\times N}$, denoted $\langle\cdot,\cdot\rangle_F$, is given by
$ \langle U,V\rangle_F := \Tr(U^*V) = \sum_{i,j=1}^NU^*_{ij}V_{ij}$.
The induced norm is known as the \emph{Frobenius norm} and is given by
  \begin{equation}
  \label{eq:fro norm}
  \|U\|_2^2=\sum_{i,j=1}^N|U_{ij}|^2= \sum_{i,j=1}^n(\Re U_{ij})^2+(\Im U_{ij})^2,
  \end{equation}
where $\Re z$ and $\Im z$ denote the real and imaginary parts of a complex number $z$, respectively.   

Given a finite set $A$ with $|A|=m$, we consider the collection $({\mathbb C}^{N\times N})^A$ of matrix-valued functions $F:A\to{\mathbb C}^{N\times N}$ which, with abusive notation, we identify with 
 $${\mathcal H}:=({\mathbb C}^{N\times N})^m:=\underbrace{{\mathbb C}^{N\times N}\times\cdots\times{\mathbb C}^{N\times N}}_{m\text{~factors}}. $$
Depending on the inner-product and field, $({\mathbb C}^{N\times N})^m$ may be viewed as a Hilbert space in two ways:
\begin{enumerate}[(a)]
  \item Over the field $\mathbb{C}$, ${\mathcal H}$ can be equipped with the inner-product $\langle\cdot,\cdot\rangle_{\mathbb{C}}$ given by
    \begin{equation}
    \label{eq:ip a}
    \langle\mathbf{U},\mathbf{V}\rangle_{\mathbb{C}} := \sum_{j=1}^{m}\langle U_j,V_j\rangle_F.
    \end{equation} 
      \item Over the field $\mathbb{R}$, ${\mathcal H}$ can be equipped with the inner-product $\langle\cdot,\cdot\rangle_{\mathbb{R}}$ given by
  \begin{equation}
  \label{eq:ip b}
  \langle\mathbf{U},\mathbf{V}\rangle_{\mathbb{R}} := \langle\Re \mathbf{U},\Re\mathbf{V}\rangle_F  + \langle \Im \mathbf{U},\Im\mathbf{V}\rangle_F = \sum_{j=1}^{m}\langle \Re U_j,\Re V_j\rangle_F  + \sum_{j=1}^{m}\langle \Im U_j,\Im V_j\rangle_F.
  \end{equation}
\end{enumerate}
Since we will only be concerned with the latter (real) inner-product, we will drop the subscript ``$\mathbb{R}$'' whenever there is no ambiguity.

\begin{proposition}
 The norms in both of the aforementioned spaces coincide.
\end{proposition}
\begin{proof}
Follows by combining \eqref{eq:fro norm}, \eqref{eq:ip a} and \eqref{eq:ip b}. 
\end{proof}

\subsection{Hilbert spaces, constraints and projections for  wavelet construction}\label{sec: constraints and projectors}
We concentrate now on the Hilbert space ${\mathcal H}=\CsigmaMn_\sigma$ and observe that the discretised wavelet construction Problem \ref{prob:discrete wavelet const prob} is equivalent to the following:

\begin{problem}\label{constraints for wavelet construction} Given an integer $M\geq 4$, find a matrix ensemble ${\mathbf U}=\{U_j\}_{j\in Q_M^n}\in \cap_{\ell =0}^{2^n-1}C_1^{(\ell )}\cap C_2\subset{\mathcal H}$ 
where the constraint sets are defined as
\begin{align*}
C_1^{(0)}&:=\{{\mathbf U}\in{\mathcal H}:\ U_j\in {\mathcal U}(2^n)\ (j\in Q_{M/2}^n\setminus\{0\}),\ U_0\in 1\otimes {\mathcal U}(2^n-1)\}\\
C_1^{(\ell )}&:=\{{\mathbf U}\in{\mathcal H}:\ ({\mathcal F}_M^{-1}\chi_\ell{\mathcal F}_M{\mathbf U})_j\in {\mathcal U}(2^n),\ (j\in Q_{M/2}^n,\ 1\leq\ell\leq 2^n-1)\}\\
C_2&:=\bigg\{{\mathbf U}\in{\mathcal H}:\ \sum_{k\in Q_M^n}c_{\alpha k}U_k\in{\mathbb C}\otimes{\mathbb C}^{(2^n-1)\times (2^n-1)}\text{ for $1\leq|\alpha |\leq d$}\bigg\}
\end{align*}
where $c_{\alpha k}=\sum_{j\in Q_M^n}j^\alpha e^{2\pi i\langle j,k\rangle /M}$.
\end{problem}

\subsubsection{Completeness and unitarity -- the constraints $C_1^{(\ell )}$}
Recall that by Proposition \ref{prop: unitary ensembles}, unitarity of the trigonometric polynomial $U(\xi )$ at all $\xi$ is equivalent to the unitarity of the ensembles $\{{\mathbf U}^{(\ell )}\}_{\ell =0}^{2^n-1}$ where ${\mathbf U}=\{U_j=U(j/M)\}_{j\in Q_M^n}$ and ${\mathbf U}^{(\ell )}={\mathcal F}_M\chi_\ell{\mathcal F}^{-1}{\mathbf U}$. Completeness requires $m_0(0)=1$ or equivalently, $(U_0)_{00}=1$. 

Let 
 $A=\left(\begin{matrix}a&{\mathbf v}^T\\{\mathbf w}&B\end{matrix}\right)\in{\mathbb C}^{2^n\times 2^n}$ with $a\in{\mathbb C}$, ${\mathbf v}$, $\mathbf w\in{\mathbb C}^{2^n-1}$ and $B\in{\mathbb C}^{(2^n-1)\times (2^n-1)}$. Then the projection $P_{1\otimes {\mathcal U}(2^n-1)}(A)$ of $A$ onto $1\otimes {\mathcal U}(2^n-1)$ is given by
$$P_{1\otimes {\mathcal U}(2^n-1)}(A)=\left(\begin{matrix}1&{\mathbf 0}^T\\{\mathbf 0}&P_{{\mathcal U}(2^n-1)}(B)\end{matrix}\right),$$
so the projection ${\mathbb P}_{C_1}^{(0)}$ from ${\mathcal H}$ onto $C_1$ is given by 
$$({\mathbb P}_{C_1}^{(0)}{\mathbf U})_j=\begin{cases}
\sigma_\ell P_{1\otimes {\mathcal U}(2^n-1)}(U_0)&\text{ if $j=Mv_\ell /2$}\\
P_{{\mathcal U}(2^n)}U_j&\text{ if $j\neq Mv_\ell /2$}
\end{cases}$$
where $P_{{\mathcal U}(2^n)}$ is the projection of ${\mathbb C}^{2^n\times 2^n}$ onto ${\mathcal U}(2^n)$ given in Proposition \ref{prop:proj examples}.

\vskip0.2in
We recall the translation operators $\tau_k$  of equation (\ref{tau def}) and define {\it modulation} operators $\mu_k$ $(k\in{\mathbb Z}^n)$ on $\CsigmaMn$ defined by
$(\mu_k{\mathbf U})_j=e^{2\pi i\langle j,k\rangle /M}U_j$.
We then have the intertwining relations
\begin{equation}
{\mathcal F}_M\tau_k=\mu_k{\mathcal F}_M;\quad {\mathcal F}_M\mu_k=\tau_{-k}{\mathcal F}_M\quad (k\in{\mathbb Z}^n)\label{trans-mod comm}
\end{equation}
and  similarly, ${\mathcal F}_M^{-1}\tau_k=\mu_{-k}{\mathcal F}_M^{-1}$ and ${\mathcal F}_M^{-1}\mu_k=\tau_k{\mathcal F}_M^{-1}$. The relationship between the modulation operators $\mu_k$ and the operator $\chi_\ell$ of Lemma \ref{lem: off-centre ensembles} is given by
$\mu_k=\prod_{\ell =1}^{n}(\chi_{2^\ell} )^{-2k_\ell}$
from which we immediately see that 
\begin{equation}
\mu_k\chi_\ell =\chi_\ell \mu_k.\label{chi-M mod}
\end{equation}
Let $S_\ell ={\mathcal F}_M\chi_\ell{\mathcal F}_M^{-1}$ $(1\leq\ell\leq n)$. Then (\ref{trans-mod comm}) and (\ref{chi-M mod}) give
\begin{equation}
S_\ell\tau_k={\mathcal F}_M\chi_\ell{\mathcal F}_M^{-1}\tau_k={\mathcal F}_M\chi_\ell \mu_k{\mathcal F}_M^{-1}={\mathcal F}_M \mu_k\chi_\ell{\mathcal F}_M^{-1}=\tau_k{\mathcal F}_M\chi_\ell{\mathcal F}_M^{-1}=\tau_k S_\ell .\label{S-tau comm}
\end{equation}
Let ${\mathbb P}_{{\mathcal U}(2^n)^{Q_M^n}}$  be the projection of $\CsigmaMn$ onto 
$${\mathcal U}(2^n)^{Q_M^n}=\{{\mathbf U}\in\CsigmaMn:\ U_j\in {\mathcal U}(2^n)\text{ for all }j\in Q_M^n\}$$
given by
$$({\mathbb P}_{{\mathcal U}(2^n)^{Q_M^n}}{\mathbf U})_j=P_{{\mathcal U}(2^n)}U_j\quad (j\in Q_M^n).$$
 For $1\leq\ell\leq 2^n-1$, consider the operator ${\mathbb P}_{C_1^{(\ell )}}:\CsigmaMn\to\CsigmaMn$ given by 
$${\mathbb P}_{C_1^{(\ell )}}{\mathbf U}=
S_{-\ell}{\mathbb P}_{{\mathcal U}(2^n)^{Q_M^n}}S_\ell .$$
We aim to show that ${\mathbb P}_{C_1^{(\ell )}}$ is the projection of ${\mathcal H}$ onto $C_1^{(\ell )}$.

Given ${\mathbf A}\in\CsigmaMn$ and $X\in{\mathbb C}^{2^n\times 2^n}$, we define $X{\mathbf A}\in\CsigmaMn$ by $(X{\mathbf A})_j=XA_j$ $(j\in Q_M^n$). Note that if $\sigma\in {\mathcal U}(2^n)$ and ${\mathbf A}\in\CsigmaMn$ then an application of (\ref{unitary proj mult comm}) gives
$$[{\mathbb P}_{{\mathcal U}(2^n)^{Q_M^n}}(\sigma{\mathbf A})]_j=P_{{\mathcal U}(2^n)}(\sigma A_j)=\sigma P_{{\mathcal U}(2^n)}(A_j)=\sigma ({\mathbb P}_{{\mathcal U}(2^n)^{Q_M^n}}{\mathbf A})_j=[\sigma {\mathbb P}_{{\mathcal U}(2^n)^{Q_M^n}}{\mathbf A}]_j$$
so that ${\mathbb P}_{{\mathcal U}(2^n)^{Q_M^n}}(\sigma {\mathbf A})=\sigma {\mathbb P}_{{\mathcal U}(2^n)^{Q_M^n}}{\mathbf A}$.

\begin{proposition}
Suppose ${\mathbf U}\in{\mathcal H}$ is an ensemble satisfying the consistency condition, i.e., $\tau_{-Mv_{2^k}/2}{\mathbf U}=\sigma_{2^k}{\mathbf U}$ $(0\leq k\leq n-1)$. Then for $1\leq\ell\leq 2^n-1$, ${\mathbb P}_{C_1^{(\ell )}}{\mathbf U}\in{\mathcal H}$, i.e., ${\mathbb P}_{C_1^{(\ell )}}$ preserves ${\mathcal H}$.
\end{proposition}
\begin{proof}
Since $\tau_n{\mathbb P}_{{\mathcal U}(2^n)^{Q_M^n}}={\mathbb P}_{{\mathcal U}(2^n)^{Q_M^n}}\tau_n$, we apply (\ref{S-tau comm}) and the consistency condition  to find
\begin{align*}
\tau_{-Mv_{2^k}/2}{\mathbb P}_{C_1^{(\ell )}}{\mathbf U}&=\tau_{-Mv_{2^k}/2}S_{-\ell}{\mathbb P}_{{\mathcal U}(2^n)^{Q_M^n}}S_\ell{\mathbf U}\\
&=S_{-\ell}\tau_{-Mv_{2^k}/2}{\mathbb P}_{{\mathcal U}(2^n)^{Q_M^n}}S_\ell{\mathbf U}\\
&=S_{-\ell}{\mathbb P}_{{\mathcal U}(2^n)^{Q_M^n}}\tau_{-Mv_{2^k}/2}S_\ell{\mathbf U}\\
&=S_{-\ell}{\mathbb P}_{{\mathcal U}(2^n)^{Q_M^n}}S_\ell\tau_{-Mv_{2^k}/2}{\mathbf U}\\
&=S_{-\ell}{\mathbb P}_{{\mathcal U}(2^n)^{Q_M^n}}S_\ell\sigma_k{\mathbf U}=\sigma_kS_{-\ell}{\mathbb P}_{{\mathcal U}(2^n)^{Q_M^n}}S_\ell{\mathbf U}=\sigma_k{\mathbb P}_{C_1^{(\ell )}}{\mathbf U}
\end{align*}
which completes the proof.
\end{proof}

\subsubsection{Regularity -- the constraint $C_2$}
An ensemble ${\mathbf U}\in\cap_{\ell =1}^{2^n-1}C_1^{(\ell )}$ may be interpreted as samples of a trigonometric polynomial $U:{\mathbb R}^n\to ({\mathbb C}^{2^n\times 2^n})^{Q_M^n}_\sigma$. In fact, if $U(\xi )=\dfrac{1}{M^n}\sum_{j\in Q_M^n}U_j\bigg(\sum_{k\in Q_M^n}e^{2\pi i\langle k,j/M-\xi\rangle}\bigg)$, then $U(\ell /M)=U_\ell$ $(\ell\in Q_M^n)$. It was shown in Section \ref{ssec: regularity} that if $U(\xi )_{j\varepsilon}=m_\varepsilon (\xi +v_j/2)$ $(0\leq j,\varepsilon\leq 2^n-1)$, $m_0(1)=1$ and $\partial^\alpha m_\varepsilon (\xi )\bigg|_{\xi =0}=0$ $(1\leq\varepsilon\leq 2^n-1,\ |\alpha |\leq d)$, then $\partial^\alpha m_0(\xi )\bigg|_{\xi =v_j/2}=0$ $(1\leq j\leq 2^n-1,\ |\alpha |\leq d)$. We conclude that if $m_0(0)=1$ then 
\begin{align*}
\sum_{k\in Q_M^n}c_{\alpha k}(U_k)_{\varepsilon , 0}&=\partial^\alpha U(\xi )_{\varepsilon , 0}\bigg|_{\xi =0}=0\quad (1\leq\varepsilon\leq 2^n-1,\ |\alpha |\leq d)\\
&\Rightarrow\partial^\alpha U(\xi )_{0,j}=\sum_{k\in Q_M^n}c_{\alpha k}(U_k)_{0,j}=0\quad (1\leq j\leq 2^n-1,\ |\alpha |\leq d)\\
&\Rightarrow\sum_{k\in Q_M^n}c_{\alpha k}U_k\in{\mathbb C}\otimes{\mathbb C}^{(2^n-1)\times (2^n-1)}\Rightarrow{\mathbf U}\in C_2.
\end{align*}
We let 
\begin{align*}
C_2'=\bigg\{{\mathbf U}\in({\mathbb C}^{2^n\times 2^n})^{Q_M^n}_\sigma:\ &\sum_{k\in Q_M^n}c_{\alpha k}U_k=\left(\begin{matrix}a_\alpha&{\mathbf 0}^T\\{\mathbf c}_\alpha&B_\alpha\end{matrix}\right)\bigg.\\
&\bigg.\text{ for some }a_\alpha\in{\mathbb C},\ c_\alpha \in{\mathbb C}^{2^n-1},\ B_\alpha\in{\mathbb C}^{(2^n-1)\times (2^n-1)}\bigg\}.
\end{align*}
Then we have shown that 
\begin{equation*}
\bigg(\cap_{\ell =0}^{2^n-1}C_1^{(\ell )}\bigg)\cap C_2=\bigg(\cap_{\ell =0}^{2^n-1}C_1^{(\ell )}\bigg)\cap C'_2
\end{equation*}
and for this reason, the constraint $C_2$ may be replaced by $C_2'$ in our algorithms.

We now consider the projection onto the subspace described by the regularity constraint $C'_2$. For $k\in Q_M^n$, define ${\mathbf w}_k\in{\mathbb C}^{2^n}$ by $({\mathbf w}_k)_\ell=(-1)^{\langle k,v_\ell\rangle}$, i.e., 
$$\big\{{\mathbf w}_k=(1,(-1)^{\langle k,v_1\rangle},\dots ,(-1)^{\langle k,v_{2^n-1}\rangle}\big)^T.$$
We note that because of the definition of the group operation $\oplus$ on $V^n$ defined in (\ref{group op}) we have that for all $k\in{\mathbb Z}^n$ and integers $0\leq j,\ell\leq 2^n-1$, $(-1)^{\langle k,v_j\oplus v_\ell\rangle}=(-1)^{\langle k,v_j\rangle}(-1)^{\langle k,v_\ell\rangle}$. Further, from the definition (\ref{sigma def nD}) of the permutation matrices $\sigma_j$, with ${\mathbf w}_k$ $(k\in Q_M^n)$ as above we have
\begin{align*}
(\sigma_j{\mathbf w}_k)_\ell&=\sum_{m=0}^{2^n-1}(\sigma_j)_{\ell m}({\mathbf w}_k)_m\\
&=\sum_{\{m:\ v_j\oplus v_\ell =v_m\}}(-1)^{\langle k, v_m\rangle}\\
&=(-1)^{\langle k,v_j\oplus v_\ell\rangle}=(-1)^{\langle k,v_j\rangle}(-1)^{\langle k,v_\ell\rangle}=(-1)^{\langle k,v_j\rangle}({\mathbf w}_k)_\ell ,
\end{align*}
so that ${\mathbf w}_k$ is an eigenvector of $\sigma_j$ with eigenvalue $(-1)^{\langle k,v_j\rangle}$.

 We work  within the Hilbert space 
\begin{align*}
{\mathcal H}_n&={\mathcal F}_M^{-1}({\mathbb C}^{2^n\times 2^n})_\sigma^{Q_M^n}\\
&=\{{\mathbf A}\in ({\mathbb C}^{2^n\times 2^n})^{Q_M^n}:\, \sigma_\ell A_k=(-1)^{\langle k,v_\ell\rangle}A_k\text{ for all $k\in Q_M^n$, $0\leq\ell\leq 2^n-1$}\}\\
&=\left\{{\mathbf A}\in ({\mathbb C}^{2^n\times 2^n})^{Q_M^n}:\,A_k={\mathbf w}_k\left(\begin{matrix}a_k&{\mathbf b}_k^T\end{matrix}\right)\right.\\
&\qquad\qquad\qquad\qquad\qquad\left.\text{ for some $a_k\in{\mathbb C}$, ${\mathbf b}_k\in{\mathbb C}^{2^n-1}$ and all $k\in Q_M^n$}\right\}
\end{align*}
so that a typical element of ${\mathcal H}_n$ is a matrix ensemble ${\mathbf A}\in({\mathbb C}^{2^n\times 2^n})^{Q_M^n}$ with $k$-th entry of the form
$$A_k=\left(\begin{matrix}a_k&{\mathbf b}_k^T\\
(-1)^{\langle k,v_1\rangle}a_k&(-1)^{\langle k,v_1\rangle}{\mathbf b}_k^T\\
\vdots&\vdots\\
(-1)^{\langle k,v_{2^n-1}\rangle}a_k&(-1)^{\langle k,v_{2^n-1}\rangle}{\mathbf b}_k^T
\end{matrix}\right)$$
for some $a_k\in{\mathbb C}$, ${\mathbf b}_k\in{\mathbb C}^{2^n-1}$.
Constraint $C_2'$ is equivalent to the condition
$\sum_{k\in Q_M^n} k^\alpha{\mathbf b}_k^T=0$ $(|\alpha |\leq d)$.
Let ${\mathcal X}_d$ be the collection of matrix ensembles ${\mathbf B}=(B_\alpha )_{|\alpha |\leq d}\subset{\mathbb C}^{2^n\times 2^n}$ of the form $B_\alpha =\left(\begin{matrix}0&\gamma_\alpha^T\\
{\mathbf 0}&\underline{\mathbf 0}\end{matrix}\right)$ where ${\mathbf 0}=(0,0,\dots ,0)\in{\mathbb C}^{2^n-1}$, $\underline{\mathbf 0}\in{\mathbb C}^{(2^n-1)\times (2^n-1)}$ is the zero matrix and $\gamma_\alpha\in{\mathbb C}^{2^n-1}$ for each $|\alpha |\leq d$. We now define an operator ${\mathcal R}:{\mathcal H}_n\to{\mathcal X}_d$ given by
\begin{equation}
({\mathcal R}{\mathbf A})_\alpha =\left(\begin{matrix}0&\sum_{k\in Q_M^n}k^\alpha{\mathbf b}_k^T\\
{\mathbf 0}&\underline{\mathbf 0}\end{matrix}\right).\label{R def}
\end{equation}
The projection we require is that onto the kernel of ${\mathcal R}$.

We decompose each $C\in{\mathbb C}^{2^n\times 2^n}$ as $C=\left(\begin{matrix}a&{\mathbf b}^T\\{\mathbf c}&D\end{matrix}\right)$ with $a\in{\mathbb C}$, ${\mathbf b}$, ${\mathbf c}\in{\mathbb C}^{2^n-1}$ and $D\in{\mathbb C}^{(2^n-1)\times (2^n-1)}$ and write $a=C_{00}$, ${\mathbf b}=C_{01}$, ${\mathbf c}=C_{10}$ and $D=C_{11}$. Let ${\mathbf A}\in{\mathcal H}_n$ and ${\mathbf B}\in{\mathcal X}_d$. We then have
\begin{align*}
\langle{\mathcal R}{\mathbf A},{\mathbf B}\rangle&=\sum_{|\alpha |\leq d}\left\langle\left(\begin{matrix}0&\sum_{k\in Q_M^n}k^\alpha{\mathbf b}_k^T\\{\mathbf 0}&\underline{\mathbf 0}\end{matrix}\right),\left(\begin{matrix}0&{\mathbf\gamma}_\alpha^T\\{\mathbf 0}&\underline{\mathbf 0}\end{matrix}\right)\right\rangle\\
&=\sum_{|\alpha |\leq d}\sum_{k\in Q_M^n}k^\alpha\big\langle{\mathbf b}_k^T,{\mathbf\gamma}_\alpha^T\big\rangle\\
&=\sum_{k\in Q_M^n}\left\langle{\mathbf b}_k^T,\sum_{|\alpha |\leq d}k^\alpha{\gamma}_\alpha^T\right\rangle\\
&=\sum_{k\in Q_M^n}\left\langle\left(\begin{matrix}a_k&{\mathbf b}_k^T\end{matrix}\right),\left(\begin{matrix}0&\sum_{|\alpha |\leq d}k^\alpha\gamma_\alpha^T\end{matrix}\right)\right\rangle\\
&=2^{-n}\sum_{k\in Q_M^n}\left\langle{\mathbf w}_k\left(\begin{matrix}a_k&{\mathbf b}_k^T\end{matrix}\right),{\mathbf w}_k\left(\begin{matrix}0&\sum_{|\alpha |\leq d}k^\alpha\gamma_\alpha^T\end{matrix}\right)\right\rangle
=\langle{\mathbf A},{\mathcal R}^*{\mathbf B}\rangle
\end{align*}
from which we conclude that 
\begin{equation}
({\mathcal R}^*{\mathbf B})_k=2^{-n}{\mathbf w}_k\left(\begin{matrix}0&\sum_{|\alpha |\leq d}k^\alpha\gamma_\alpha^T\end{matrix}\right)=
2^{-n}\left(\begin{matrix}0&\sum_{|\alpha |\leq d}k^\alpha\gamma_\alpha^T\\
0&(-1)^{\langle k,v_1\rangle}\sum_{|\alpha |\leq d}k^\alpha\gamma_\alpha^T\\
0&(-1)^{\langle k,v_2\rangle}\sum_{|\alpha |\leq d}k^\alpha\gamma_\alpha^T\\
\vdots&\vdots\\
0&(-1)^{\langle k,v_{2^n-1}\rangle}\sum_{|\alpha |\leq d}k^\alpha\gamma_\alpha^T\end{matrix}\right).\label{R* def}
\end{equation}
Let $C_{n,d}=|\{\alpha\in({\mathbb N}\cup\{0\})^n:\, |\alpha |\leq d\}|$. For example, $C_{2,d}=(d+1)(d+1)/2$. If ${\mathbf B}=(B_\alpha )_{|\alpha |\leq d}\in{\mathcal X}_d$, i.e, $B_\alpha=\left(\begin{matrix}0&\gamma_\alpha^T\\{\mathbf 0}&\underline{\mathbf 0}\end{matrix}\right)$, then 
$$({\mathcal R}{\mathcal R}^*{\mathbf B})_\beta=2^{-n}\left(\begin{matrix}0&\sum_{|\alpha |\leq d}G_{\beta\alpha}\gamma_\alpha^T\\{\mathbf 0}&\underline{\mathbf 0}\end{matrix}\right)$$
where $G\in{\mathbb C}^{C_{n,d}\times C_{n,d}}$ has $(\beta ,\alpha )$-th entry $G_{\beta\alpha}=\sum_{k\in Q_M^n}k^{\alpha +\beta}$. We wish to show that ${\mathcal R}{\mathcal R}^*$ is invertible. Consider functions $r_\alpha :Q_M^n\to{\mathbb C}$ given by $r_\alpha (k)=k^\alpha$. We claim that $\{r_\alpha\}_{|\alpha |\leq d}$ is a linearly independent set. To see this, suppose there are constants $\{a_\alpha\}_{|\alpha |\leq d}\subset{\mathbb C}$ such that $\sum_{|\alpha |\leq d}a_\alpha r_\alpha =0$, i.e., $\sum_{|\alpha |\leq d}a_\alpha k^\alpha =0$ for all $k\in Q_M^n$. Let $p(x)=\sum_{|\alpha |\leq d}a_\alpha x^\alpha$ $(x\in{\mathbb R}^n)$. Then $p$ is a (multivariate) polynomial of degree less than or equal to $d$ and $p(k)=0$ for all $k\in Q_M^n$. Hence $p\equiv 0$, i.e., $a_\alpha=0$ for all $\alpha$. We conclude that $\{r_\alpha\}_{|\alpha |\leq d}$ is a linearly independent set. Suppose now that ${\mathbf a}=(a_\alpha)_{|\alpha |\leq d}$  is such that $G{\mathbf a}={\mathbf 0}$. Then 
\begin{equation}
0=(G{\mathbf a})_\beta=\sum_{|\alpha |\leq d}G_{\beta\alpha}a_\alpha=\sum_{k\in Q_M^n}k^\beta\sum_{|\alpha |\leq d}a_\alpha k^\alpha=\bigg\langle r_\beta ,\sum_{|\alpha |\leq d}a_\alpha r_\alpha\bigg\rangle .\label{LI}
\end{equation}
But $\sum_{|\alpha |\leq d}a_\alpha r_\alpha\in S_d=\text{sp}\{r_\beta\}_{|\beta |\leq d}$ and $\{r_\beta\}_{|\beta |\leq d}$ is a basis for $S_d$, so by (\ref{LI}) we conclude that $\sum_{|\alpha |\leq d}a_\alpha r_\alpha\equiv 0$, or equivalently, $p(k)=0$ for all $k\in Q_M^n$ where $p(x)=\sum_{|\alpha |\leq d}a_\alpha x^\alpha$. Hence $p\equiv 0$ and $a_\alpha =0$ for all $\alpha$, i.e., $G$ is invertible.
We then have
\begin{equation}
(({\mathcal R}{\mathcal R}^*)^{-1}{\mathbf B})_\beta=2^n\left(\begin{matrix}0&\sum_{|\alpha |\leq d}G^{-1}_{\beta\alpha}\gamma^T_\alpha\\
{\mathbf 0}&\underline{\mathbf 0}\end{matrix}\right).\label{RR* def}
\end{equation}
Combining (\ref{R def}), (\ref{R* def}) and (\ref{RR* def}) gives
$$({\mathcal R}^*({\mathcal R}{\mathcal R}^*)^{-1}{\mathcal R}{\mathbf A})_k=\sum_{|\alpha |\leq d}k^\alpha\sum_{|\beta |\leq d}G^{-1}_{\alpha\beta}\sum_{\ell\in Q_M^n}\ell^\beta\left(\begin{matrix}0&{\mathbf b}_\ell^T\\{\mathbf 0}&\underline{\mathbf 0}\end{matrix}\right)$$
so that the projection $Q$ of an ensemble ${\mathbf A}\in{\mathcal F}_M^{-1}\CsigmaMn$ onto ${\mathcal F}_M^{-1}C'_2$ is given by 
\begin{align*}
(Q{\mathbf A})_k&=A_k-({\mathcal R}^*({\mathcal R}{\mathcal R}^*)^{-1}{\mathcal R}{\mathbf A})_k\\
&={\mathbf w}_k\left(\begin{matrix} a_k&{\mathbf b}_k^T-\sum_{|\alpha |\leq d}\sum_{|\beta |\leq d}\sum_{\ell\in Q_M^n}G^{-1}_{\alpha\beta}k^\alpha\ell^\beta{\mathbf b}_k^T\end{matrix}\right).
\end{align*}
Finally, the required projection $P_{C'_2}$ of $\CsigmaMn$ onto $C'_2$ is given by 
\begin{align}
&(P_{C'_2}{\mathbf U})_j=({\mathcal F}_MQ{\mathcal F}_M^{-1}{\mathbf U})_j\notag\\
&=U_j-\sum_{k\in Q_M^n}(Q{\mathcal F}_M^{-1}{\mathbf U})_ke^{2\pi i\langle k,j\rangle /M}\notag\\
&=U_j-\sum_{k\in Q_M^n}({\mathcal R}^*({\mathcal R}{\mathcal R}^*)^{-1}{\mathcal R}({\mathcal F}_M^{-1}{\mathbf U}))_ke^{2\pi i\langle k,j\rangle/M}\notag\\
&=U_j-\frac{1}{M^n}\sum_{k\in Q_M^n}e^{2\pi i\langle j,k\rangle /M}\sum_{|\alpha |\leq d}k^\alpha\sum_{|\beta |\leq d}G^{-1}_{\alpha\beta}\left(\begin{matrix}0&\sum_{\ell\in Q_M^n}\ell^\beta ({\mathcal F}_M^{-1}{\mathbf U})^T_{\ell, 01}\\{\mathbf 0}&\underline{\mathbf 0}\end{matrix}\right).\label{proj C_2'}
\end{align}
Let $C\in {\mathbb C}^{R^n_d\times Q_M^n}$ have $(\beta ,m)$-th entry $c_{\beta m}=\sum_{\ell\in Q_M^n}\ell^\beta e^{-2\pi i\langle m,\ell\rangle /M}$. Then (\ref{proj C_2'}) may be written as
$$({\mathbb P}_{C'_2}{\mathbf U})_j=U_j-\frac{1}{M^n}\sum_{m\in Q_M^n}(C^*G^{-1}C)_{jm}\left(\begin{matrix}0&(U_m)_{01}\\{\mathbf 0}&\underline{\mathbf 0}\end{matrix}\right).$$

\section{Computational Results: one- and two-dimensional wavelets}\label{sec: comp}
In this section, we report representative computational results for the DR algorithm (as described in Algorithm~\ref{alg:DR}) applied to the formulations described in Problem~\ref{constraints for wavelet construction}.\footnote{The accompanying source code is available at \href{https://gitlab.com/matthewktam/drwavelets}{https://gitlab.com/matthewktam/drwavelets}.}
The main goal of reporting these results is to provide an insight into the typical number of iterations and the success rate of the method for the wavelet reconstruction problem. All experiments implemented in Python 3.7 and a machine having an Intel Xeon E5-4650 @ 2.70GHz running Red Hat Enterprise Linux 3.10.

For each value of $(M,d)$ examined, ten replications of the DR algorithm were run, each starting from a different randomly generated initialisation $x_0\in D$, where $D$ denotes the diagonal subspace from Section~\ref{sec: projection algorithms}. More precisely, the real and complex entries, respectively, of a matrix ensemble $\mathbf{U}^0\in(\mathbb{C}^{2^n\times 2^n})^{Q^n_M}$ were generated entry-wise by sampling from the uniform distribution on the interval $(-1,+1)$. The ensemble $\mathbf{U}_0$ was then projected onto $C_1^{(0)}$, and its projection was then used to form the tuple of ensembles $x_0$.

The algorithm was terminated if either: (i) the stopping criterion 
 $$ \|x_k-x_{k+1}\| < \epsilon $$
was satisfied with $\epsilon=10^{-3}$, or (ii) more than $10^{6}$ iterations had been performed. In the case that the algorithm terminated successfully (i.e., the stopping criterion was satisfied), orthogonality of the resulting trigonometric polynomial was checked numerically using Bownik's condition as described in Section~\ref{ssec: regularity}. For the 2D problem, non-separability was also checked using the procedure outlined in Section~\ref{ssec: nonseparability}.

Tables~\ref{t:1d results}  and \ref{t:2d results} report a summary of the results for the 1D and 2D problems, respectively. In addition to the number of instances solved (out of ten), the mean number of iterations and time in seconds for solved instances are shown. The maxima across solved instances are also shown in parentheses. The mean and (in parentheses) maximum separability measure of solved examples is shown in the final column of table \ref{t:2d results}.

Exemplar results are provided in Figures~\ref{f:R6final}--\ref{f:2dneil}. The two-dimensional scaling function and wavelets of Figure \ref{f:2dfinal} and associated filters pass Bownik's test (Theorem \ref{thm: regularity-bownik}) for orthogonality and the separability measure (see Section \ref{ssec: nonseparability}) of the filter coefficient matrix $G^0$ is $0.041$. This compares poorly with the average separability $(0.438)$ of random matrices $H=(h_{jk})_{j,k=0}^5$ satisfying the conditions
$$\sum_{j,k=0}^5h_{jk}=1;\qquad \sum_{j,k=0}^5|h_{jk}|^2=\frac{1}{2}.$$
Of course, these filters do not satisfy the extra regularity, consistency, or unitarity conditions satisfied by the filter given in Figure~\ref{f:2dfinal}(c). Nevertheless, real-valued scaling functions have been generated by the algorithm described in this paper with $M=6$, $d=2$ and relatively high non-separability. An example is given in Figure~\ref{f:2dneil}. The separability measure of this example is approximately $0.315$. 

Further constraints designed to force real-valuedness of multidimensional scaling functions and wavelets and  to promote symmetry and cardinality are imposed on matrix ensembles in \cite {DHT}. 

\begin{table}[htb]
	\caption{Mean (worst case) results from $10$ replications for the 1D problem with $\epsilon=10^{-3}$.\label{t:1d results}}
	\centering
	\begin{tabular}{|c|c|rr|rr|} \hline
	  $(M,d)$  & Solved   & \multicolumn{2}{c|}{Iterations}    & \multicolumn{2}{c|}{Time (s)} \\ \hline
	  $(4,1)$  & 10                &     122.2 &     (162)     &     0.1 &      (0.2) \\
	  $(6,2)$  & 9                 &   3 852.0 &   (9 361)     &     5.1 &     (12.5) \\  
	  $(8,3)$  & 10                &  40 672.5 & (112 460)     &    67.6 &    (186.7) \\
	  $(10,4)$ & 8                 & 154 372.8 & (607 495)     &   325.8 &  (1 280.5) \\
	  $(12,5)$ & 9                 & 166 251.0 & (369 136)     &   422.0 &    (932.8) \\
	  $(14,6)$ & 6                 & 302 014.3 & (690 650)     &   917.1 &  (2 093.7) \\ \hline
\end{tabular}	  
\end{table}

\begin{figure}[hptb]
\begin{subfigure}[p]{0.99\textwidth}
	\centering
	\vspace{1.25em}
	\includegraphics[width=0.95\textwidth]{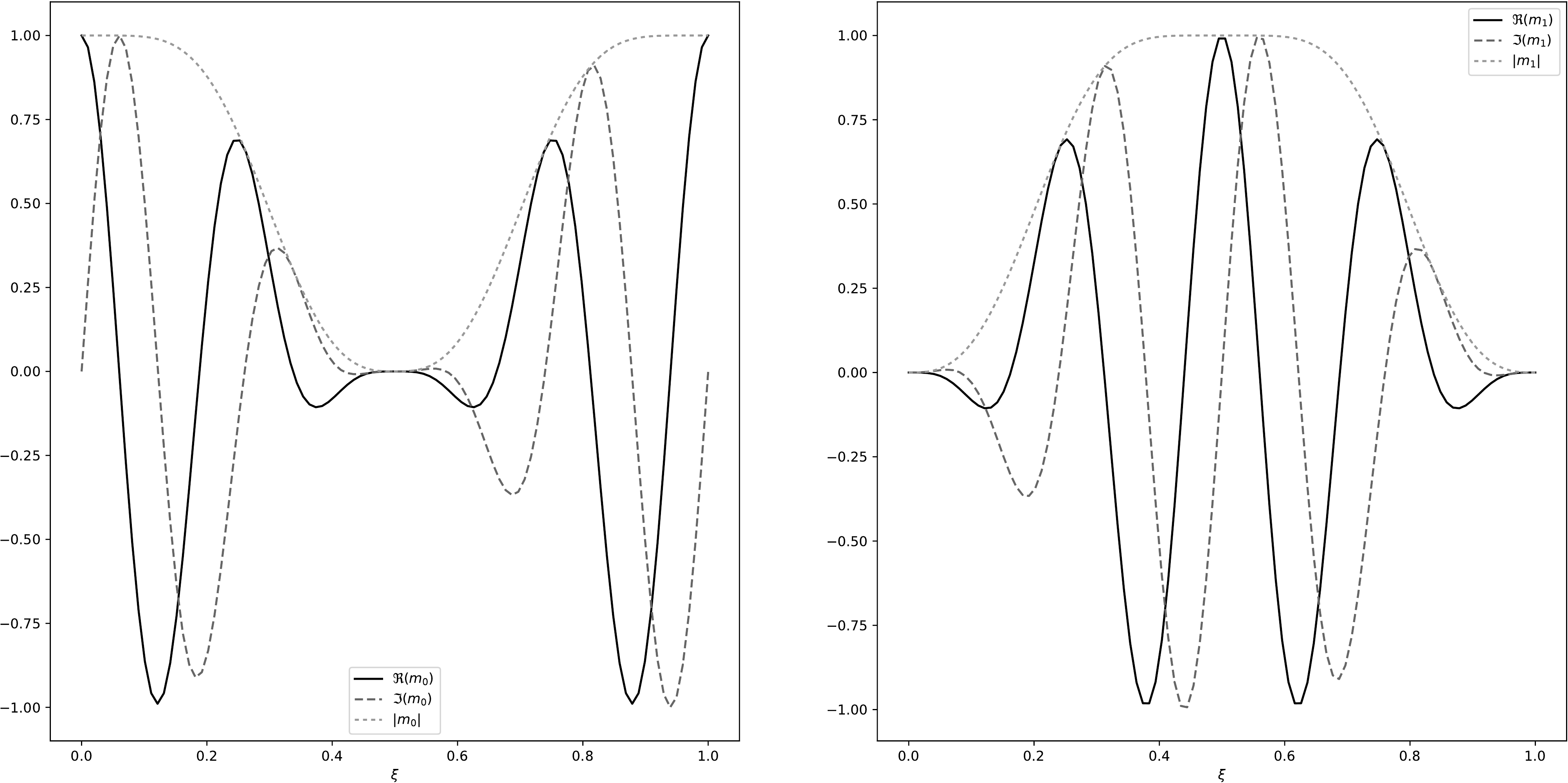}
	\caption{The polynomials $m_0(\xi)$ and $m_1(\xi)$, respectively.}
\end{subfigure}	
	
\bigskip	

\begin{subfigure}[p]{0.49\textwidth}
\centering
\includegraphics[height=18em]{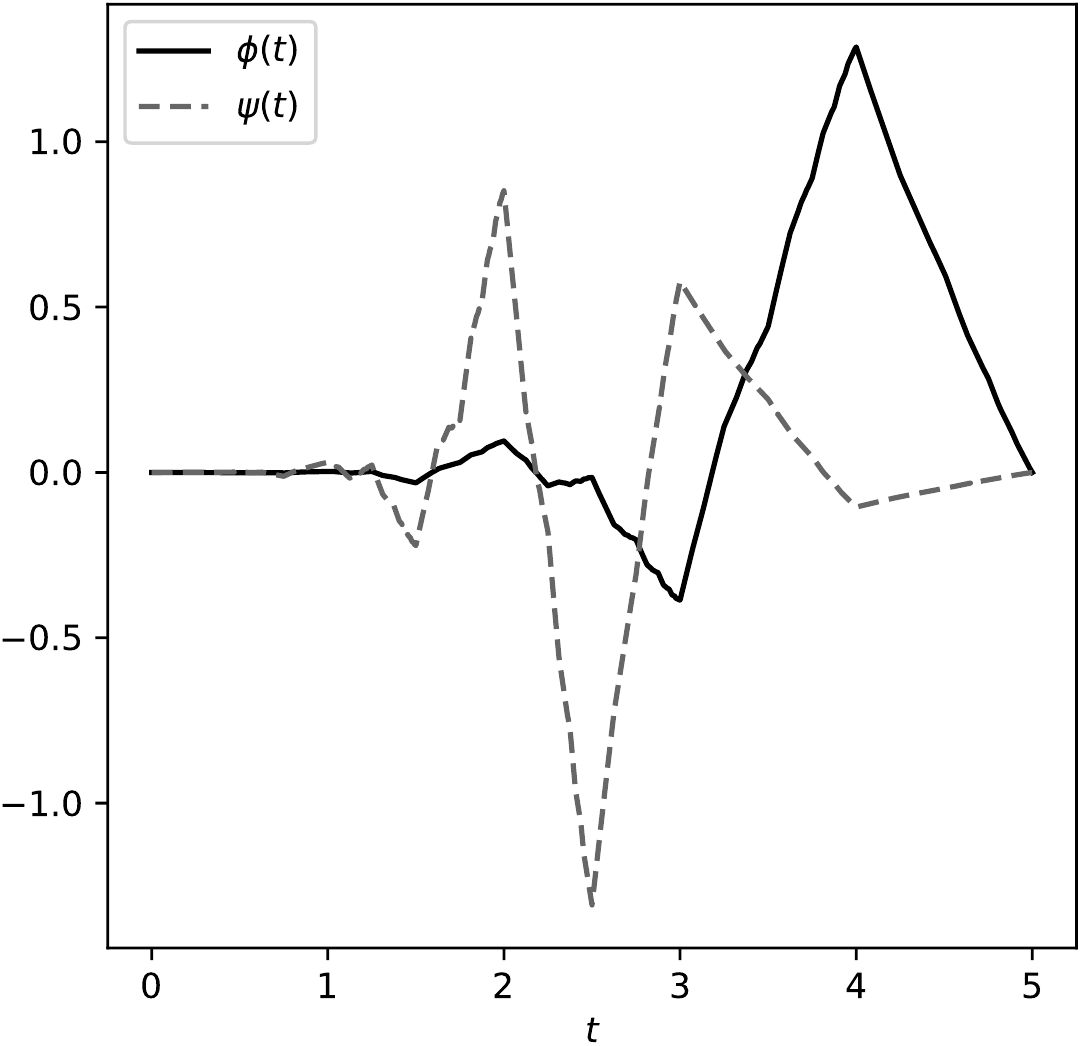}
\caption{The scaling function $\varphi$ and wavelet $\psi$.}
\end{subfigure}	
\begin{subfigure}[p]{0.49\textwidth}
\centering
\vspace{12.75ex}
\begin{minipage}[c]{0.6\textwidth}

\begin{verbatim}
	    [  0.02490875,
	      -0.0604161 ,
	      -0.09546721,
	       0.3251825 ,
	       0.57055846,
	       0.2352336  ].
\end{verbatim}

\end{minipage}
\vspace{12.75ex}
\caption{The corresponding coefficients.}
\end{subfigure}	
\caption{An exemplar 1D result obtained from the DR algorithm for $(M,d)=(6,2)$.\label{f:R6final}}
\end{figure}

\begin{figure}[hptb]
\begin{subfigure}[p]{0.99\textwidth}
	\centering
	\vspace{1.25em}
	\includegraphics[width=0.95\textwidth]{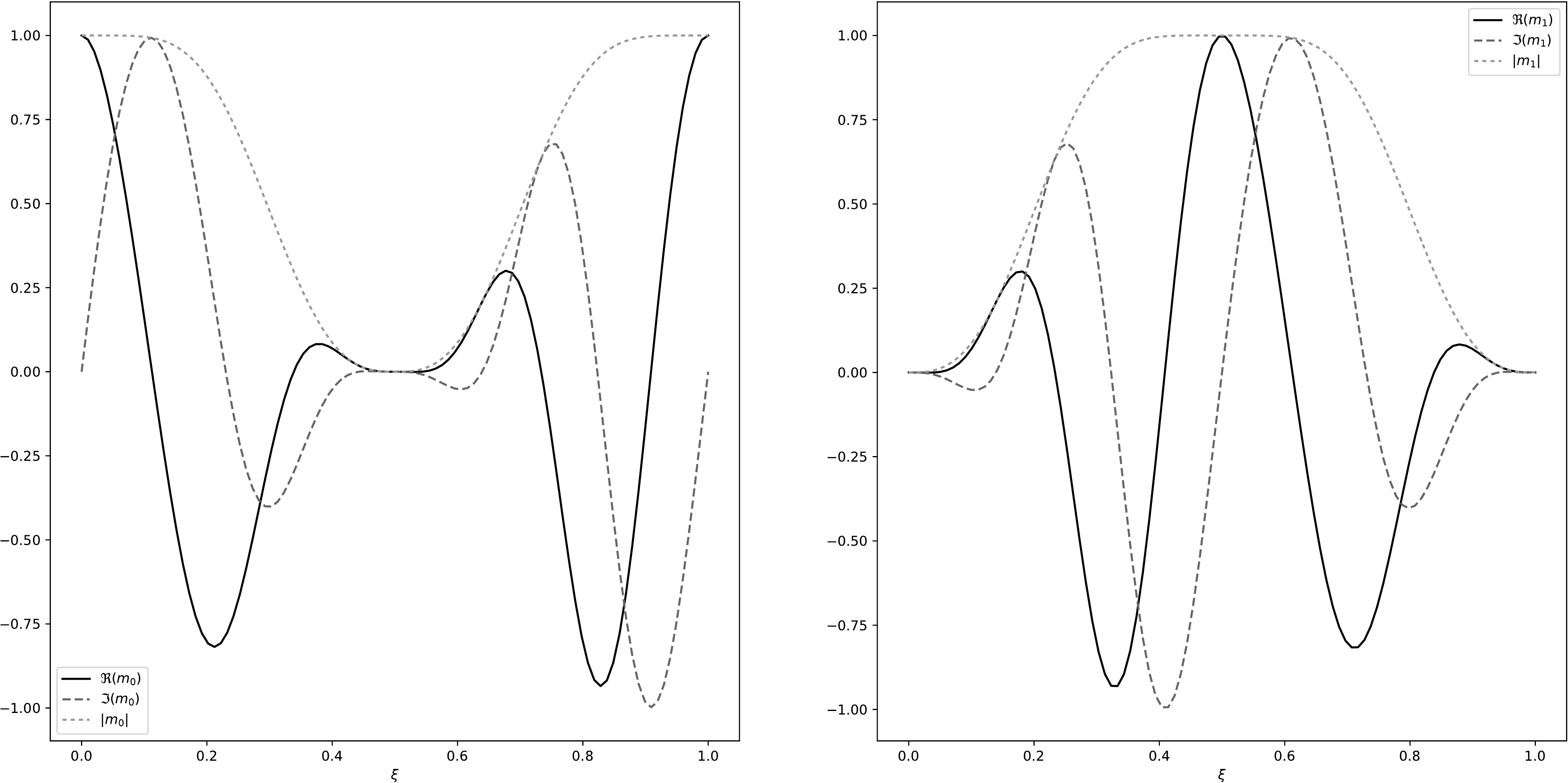}
	\caption{The polynomials $m_0(\xi)$ and $m_1(\xi)$, respectively.}
\end{subfigure}	
	
\bigskip	

\begin{subfigure}[p]{0.99\textwidth}
\centering
\includegraphics[width=0.95\textwidth]{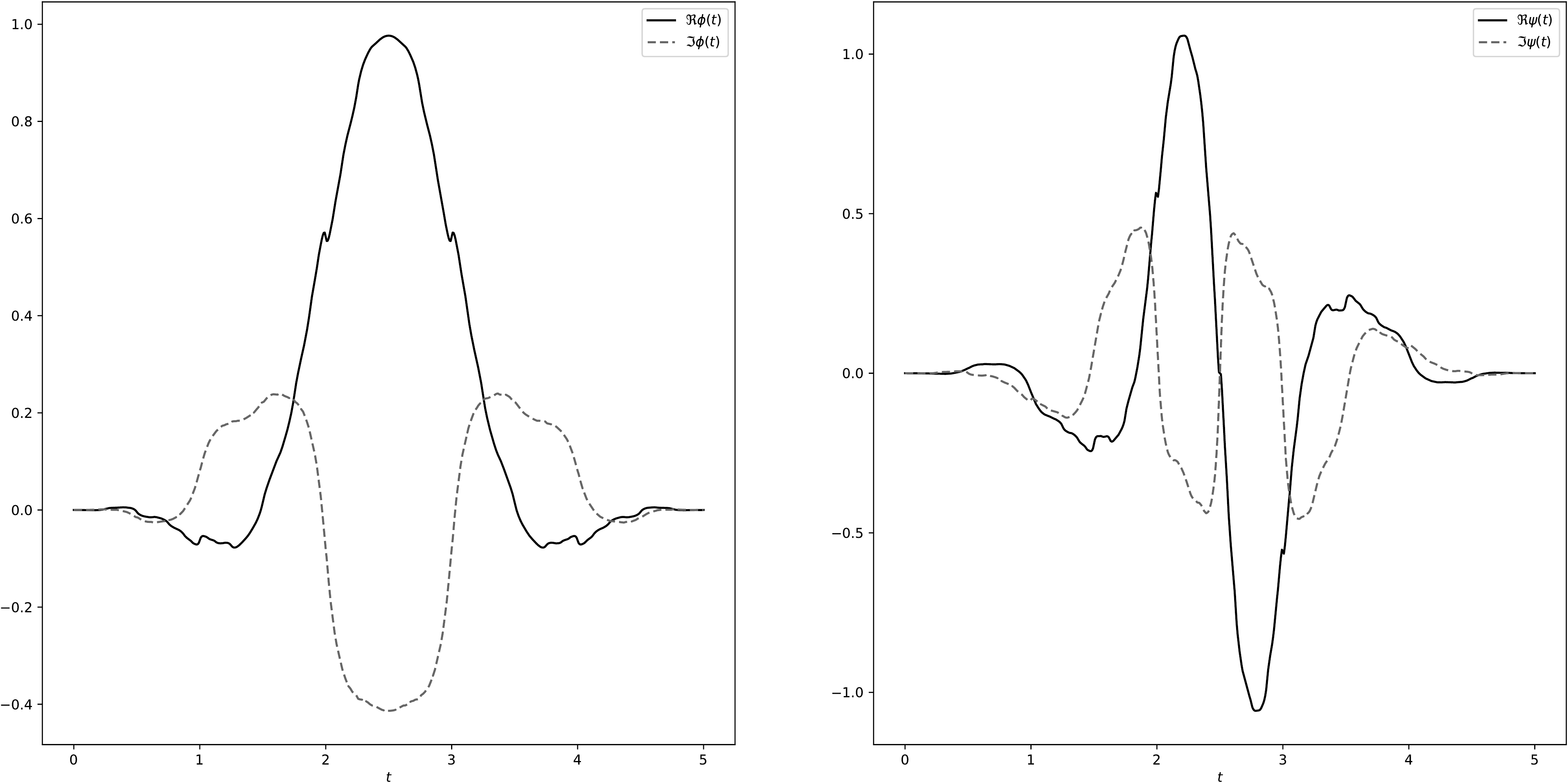}
\caption{The scaling function $\varphi$ and wavelet $\psi$.}
\end{subfigure}	
\begin{subfigure}[p]{0.99\textwidth}
\centering

\bigskip

\begin{minipage}[c]{0.7\textwidth}
{\small \begin{verbatim}
[-0.046875+0.060515i  0.078125+0.060515i  0.46875 -0.1210307i
  0.468750-0.121030i  0.078125+0.060515i -0.046875+0.0605153i]  
\end{verbatim}}

\end{minipage}
\caption{The corresponding coefficients.}
\end{subfigure}	
\caption{An exemplar 1D result obtained from the DR algorithm for $(M,d)=(6,2)$.\label{f:M6d2-s1953}}
\end{figure}

\begin{table}[htb]
	\centering
	\caption{Mean (worst case) results from $10$ replications for the 2D problem with $\epsilon=10^{-3}$.\label{t:2d results}}
	\begin{tabular}{|c|c|rr|rr|rr|} \hline
	  $(M,d)$  & Solved   & \multicolumn{2}{c|}{Iterations}    & \multicolumn{2}{c|}{Time (s)} & \multicolumn{2}{c|}{$S(\varphi)$}\\ \hline
	  $(4,1)$ & 10 &   4 469.2 &  (28 387) &    118.9 &    (708.2) & 0.209 & (0.250) \\
	  $(6,2)$ &  6 & 180 864.3 & (747 870) & 22 322.2 & (92 288.1) & 0.104 & (0.207) \\ \hline
\end{tabular}	
\end{table}

\begin{sidewaysfigure}[hptb]
\begin{subfigure}[p]{0.49\textwidth}
	\centering
	\vspace{1.25em}
	\includegraphics[height=0.35\textheight]{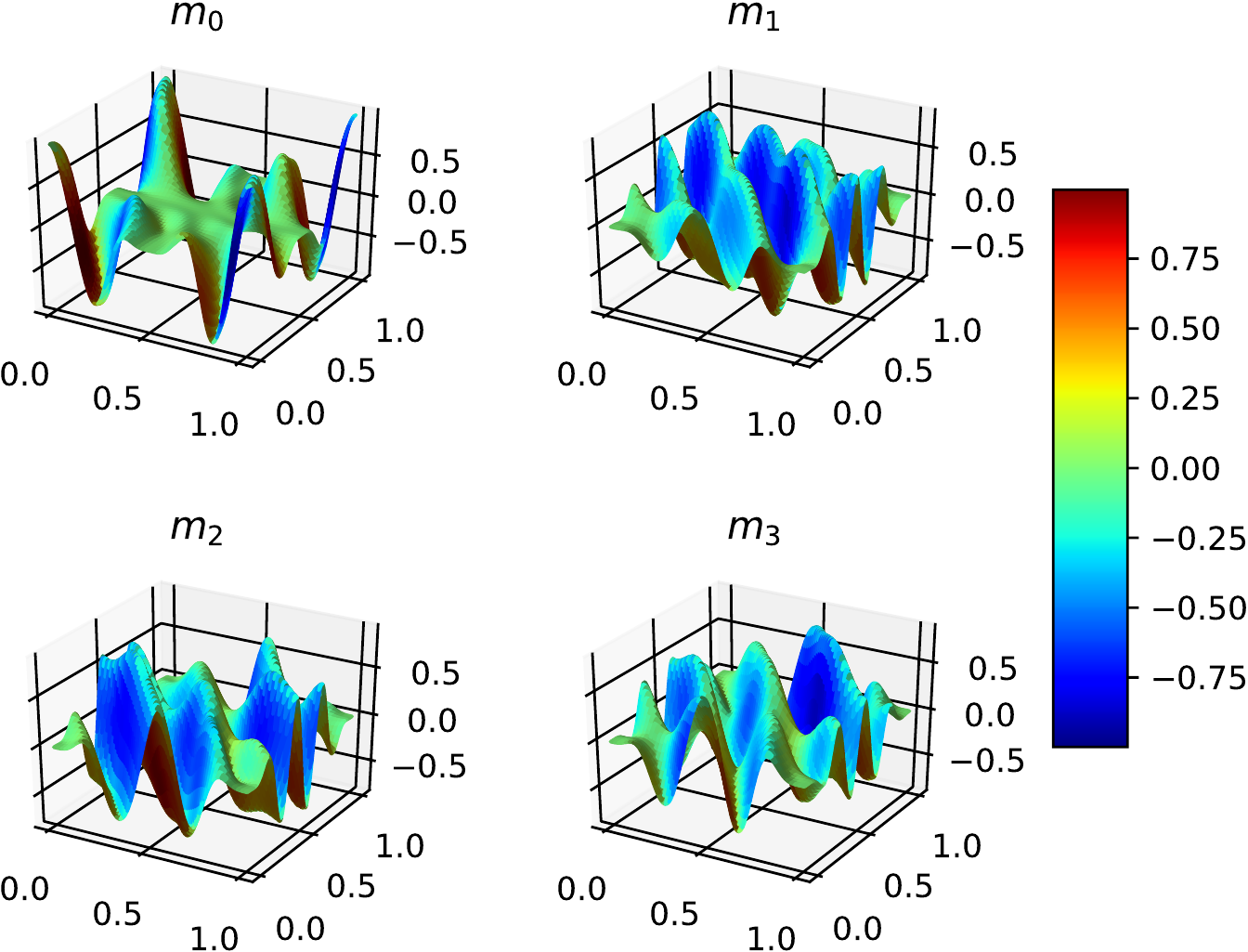}
	\caption{The polynomials $m_0(\xi)$, $m_1(\xi)$,  $m_2(\xi)$ and  $m_3(\xi)$ respectively. The real components of the polynomials are shown on the vertical axes. The imaginary components are represented by colour. }
\end{subfigure}	
\begin{subfigure}[p]{0.49\textwidth}
\centering
\includegraphics[height=0.35\textheight]{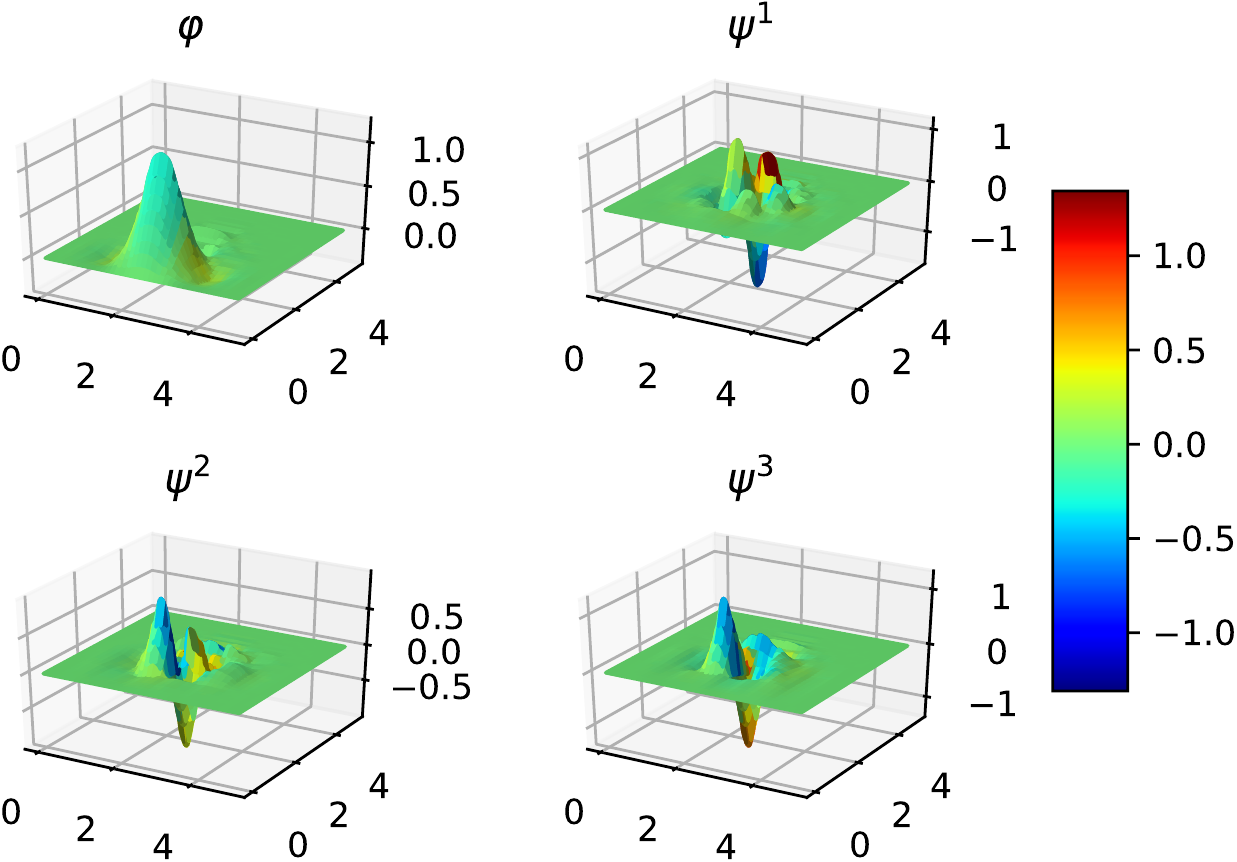}
\caption{The scaling function $\varphi$ and the wavelets $\psi^\varepsilon$. The real components of the functions are shown on the vertical axes. The imaginary components are represented by colour.}
\end{subfigure}	
\begin{subfigure}[p]{0.99\textwidth}
\centering
\vspace{5.75ex}
\begin{minipage}[c]{1\textwidth}

{\scriptsize
\begin{verbatim}
[[-1.315e-02+1.780e-02i -2.573e-02+3.276e-02i -1.046e-02+1.251e-02i  1.927e-03-2.124e-03i  1.794e-04-5.040e-05i  3.696e-04-3.775e-04i]
 [ 2.368e-02+1.746e-02i  4.149e-02+3.268e-02i  1.548e-02+1.301e-02i -1.962e-03-1.942e-03i -9.656e-05-2.155e-04i -4.618e-04-4.822e-04i]
 [ 1.073e-01-3.551e-02i  2.693e-01-6.537e-02i  1.573e-01-2.520e-02i -4.749e-02+3.955e-03i -3.024e-02+1.919e-04i  1.252e-02+9.006e-04i]
 [ 1.070e-01-3.500e-02i  2.696e-01-6.551e-02i  1.575e-01-2.585e-02i -4.774e-02+4.176e-03i -3.019e-02+3.426e-04i  1.249e-02+8.180e-04i]
 [ 2.347e-02+1.771e-02i  4.167e-02+3.261e-02i  1.574e-02+1.269e-02i -2.168e-03-1.831e-03i -1.512e-04-1.415e-04i -4.363e-04-5.232e-04i]
 [-1.310e-02+1.755e-02i -2.583e-02+3.283e-02i -1.042e-02+1.284e-02i  1.965e-03-2.233e-03i  7.934e-05-1.271e-04i  4.278e-04-3.358e-04i]]
\end{verbatim}
}
\end{minipage}
\vspace{5.75ex}
\caption{The corresponding scaling function coefficients.}
\end{subfigure}	
\caption{An exemplar 2D result obtained from the DR algorithm for $(M,d)=(6,2)$.\label{f:2dfinal}}
\end{sidewaysfigure}

\begin{sidewaysfigure}[hptb]
\begin{subfigure}[p]{0.49\textwidth}
	\centering
	\vspace{1.25em}
	\includegraphics[height=0.35\textheight]{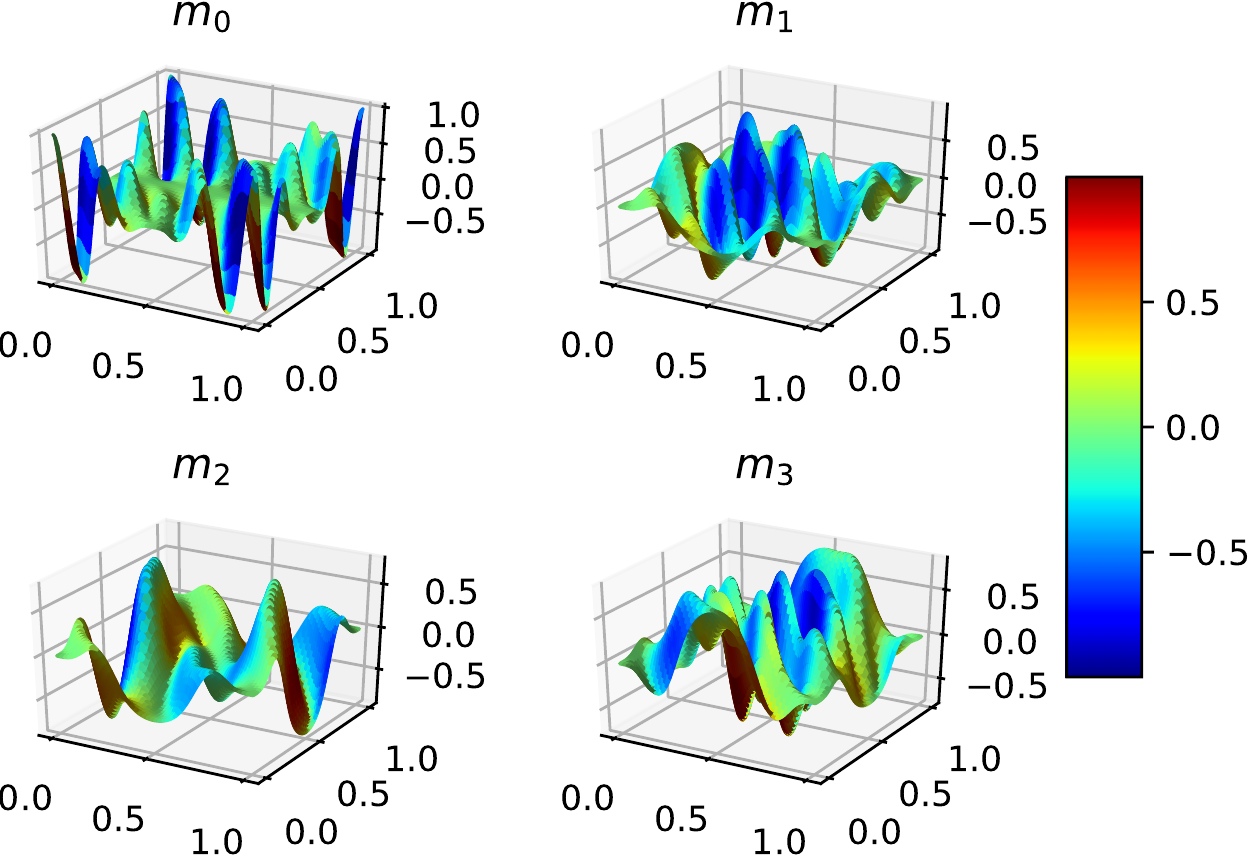}
	\caption{The polynomials $m_0(\xi)$, $m_1(\xi)$,  $m_2(\xi)$ and  $m_3(\xi)$ respectively. The real components of the polynomials are shown on the vertical axes. The imaginary components are represented by colour. }
\end{subfigure}	
\begin{subfigure}[p]{0.49\textwidth}
\centering
~\hspace{1.25cm}\includegraphics[height=0.35\textheight]{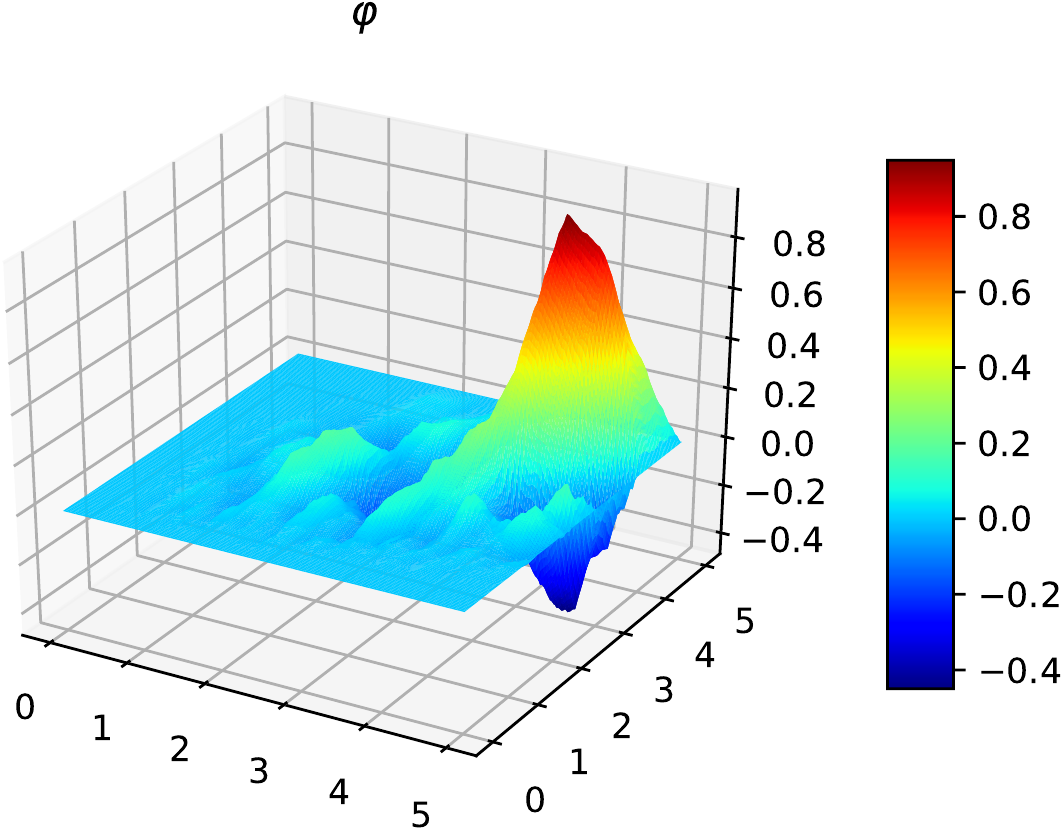}
\caption{The (real-valued) scaling function $\varphi$.}
\end{subfigure}	
\begin{subfigure}[p]{0.99\textwidth}
\centering
\vspace{5.75ex}
\begin{minipage}[c]{0.4\textwidth}
{\scriptsize
\begin{verbatim}
[[ 0.0368,  0.0406, -0.0305, -0.0362,  0.0061,  0.008 ],
 [-0.0341, -0.0591, -0.0079,  0.0292,  0.0118, -0.0003],
 [-0.0308, -0.0405, -0.0225, -0.0185,  0.0056,  0.0113],
 [ 0.0304,  0.0661,  0.082 ,  0.0877,  0.0502,  0.0088],
 [ 0.0065, -0.0303,  0.0052,  0.2172,  0.2736,  0.0984],
 [ 0.0161, -0.0371, -0.1218,  0.0457,  0.2233,  0.1091]]
\end{verbatim}
}
\end{minipage}
\vspace{5.75ex}
\caption{The corresponding scaling function coefficients.}
\end{subfigure}	
\caption{An exemplar 2D result obtained from the DR algorithm for $(M,d)=(6,2)$.\label{f:2dneil}}
\end{sidewaysfigure}

\section*{Acknowledgement} The authors are grateful for the input of Neil Dizon who helped in the generation of the figures and provided the highly non-separable example of Section \ref{sec: comp}.

\noindent JAH was supported by the Australian Research Council through DP160101537. MKT was supported by the Australian Research Council through DE200100063. Thanks Roy. Thanks HG.

\end{document}